\pdfoutput=1
\RequirePackage{ifpdf}
\ifpdf 
\documentclass[pdftex]{sigma}
\else
\documentclass{sigma}
\fi

\numberwithin{equation}{section}

\usepackage{bbm}
\usepackage{lineno}
\usepackage{cancel}
	
\newcommand{\Z}{\mathbb Z}

\newcommand{\C}{\mathbb C}

\newcommand{\F}{\mathcal F}
\newcommand{\W}{\mathcal W}

\newcommand{\bea}{\begin{eqnarray}}
\newcommand{\eea}{\end{eqnarray}}

\newtheorem{Theorem}{Theorem}[section]
\newtheorem*{Theorem*}{Theorem}
\newtheorem{Corollary}[Theorem]{Corollary}
\newtheorem{Lemma}[Theorem]{Lemma}
\newtheorem{Proposition}[Theorem]{Proposition}
\newtheorem{conj}[Theorem]{Conjecture}
\newtheorem{Claim}[Theorem]{Claim}

 { \theoremstyle{definition}
\newtheorem{Definition}[Theorem]{Definition}

\newtheorem{Remark}[Theorem]{Remark}
}

\begin{document}
\allowdisplaybreaks

\newcommand{\arXivNumber}{2411.08406}

\renewcommand{\thefootnote}{}

\renewcommand{\PaperNumber}{036}

\FirstPageHeading

\ShortArticleName{Kazama--Suzuki Duality between $\mathcal W_k(\mathfrak{sl}_4, f_{\rm sub})$ and $N=2$ Superconformal Vertex Algebra}

\ArticleName{On Kazama--Suzuki Duality between $\boldsymbol{\mathcal W_k(\mathfrak{sl}_4, f_{\rm sub})}$\\ and $\boldsymbol{N=2}$ Superconformal Vertex Algebra\footnote{This paper is a~contribution to the Special Issue on Recent Advances in Vertex Operator Algebras in honor of James Lepowsky. The~full collection is available at \href{https://www.emis.de/journals/SIGMA/Lepowsky.html}{https://www.emis.de/journals/SIGMA/Lepowsky.html}}}

\Author{Dra\v{z}en ADAMOVI\'C~$^{\rm a}$ and Ana KONTREC~$^{\rm b}$}

\AuthorNameForHeading{D.~Adamovi\'c and A.~Kontrec}

\Address{$^{\rm a)}$~Department of Mathematics, Faculty of Science, University of Zagreb,\\
\hphantom{$^{\rm a)}$}~Bijeni\v{c}ka 30, Zagreb, Croatia}
\EmailD{\href{mailto:adamovic@math.hr}{adamovic@math.hr}}
\URLaddressD{\url{ https://web.math.pmf.unizg.hr/~adamovic/}}

\Address{$^{\rm b)}$~Research Institute for Mathematical Sciences (RIMS), Kyoto University,\\
\hphantom{$^{\rm b)}$}~Kyoto 606-8502, Japan}
\EmailD{\href{mailto:akontrec@kurims.kyoto-u.ac.jp}{akontrec@kurims.kyoto-u.ac.jp}}

\ArticleDates{Received December 31, 2024, in final form May 08, 2025; Published online May 17, 2025}

\Abstract{We classify all possible occurrences of Kazama--Suzuki duality between the ${N=2}$ superconformal algebra $L^{N=2}_c$ and the subregular $\mathcal{W}$-algebra $\mathcal{W}_{k}(\mathfrak{sl}_4, f_{\rm sub})$. We establish a new Kazama--Suzuki duality between the subregular $\mathcal{W}$-algebra $\mathcal{W}_k(\mathfrak{sl}_4, f_{\rm sub})$ and the $N = 2$ superconformal algebra $L^{N=2}_{c}$ for $c=-15$. As a consequence of the duality, we classify the irreducible $\mathcal{W}_{k=-1}(\mathfrak{sl}_4, f_{\rm sub})$-modules.}


\Classification{17B69; 17B67; 17B68}

\begin{flushright}
\begin{minipage}{58mm}
\it Dedicated to James Lepowsky\\ on the occasion of his 80th birthday
\end{minipage}
\end{flushright}

\renewcommand{\thefootnote}{\arabic{footnote}}
\setcounter{footnote}{0}

\section{Introduction}

{\bf Kazama--Suzuki coset construction and vertex algebra reformulation.} In the paper~\cite{KS}, Kazama and Suzuki described a construction of the $N=2$ superconformal algebra, including $N=2$ superconformal minimal models, using the coset space method \cite{GKO}. They determined under which conditions the $N=1$ super-GKO coset construction could be extended to the $N = 2$ superconformal algebra.

In an effort to better understand the highest weight-type modules for the $N = 2$ algebra, Feigin, Semikhatov and Tipunin showed that certain categories of modules of the $N = 2$ superconformal algebra $L^{N=2}_{c}$ and affine vertex algebra $L_s(\mathfrak{sl}_2)$ are equivalent by constructing the so-called Kazama--Suzuki (KS) and inverse Kazama--Suzuki mapping~\cite{FST}.

The first named author interpreted in \cite{A-IMRN} the Kazama--Suzuki \cite{KS} and inverse Kazama--Suzuki mapping \cite{FST} in the language of vertex algebras as embeddings between certain simple vertex operator algebras (VOAs). In vertex algebra terms, the $N = 2$ algebra $L^{N=2}_c$ is realized as a~coset~${\operatorname{Com} (\mathcal{H}, L_s(\mathfrak{sl}_2) \otimes \F_1)}$, where $\mathcal H$ is a certain Heisenberg subalgebra and $\F_1$ is a lattice vertex algebra. As a consequence of the duality, he obtained a complete classification of irreducible $L^{N=2}_{c_m}$-modules for admissible~$m$. The Kazama--Suzuki duality was used in \cite{A-2001} for the proof of rationality and regularity of $L^{N=2}_{c_m}$ for~${m \in {\Z}}$, and determination of the fusion rules.
 Extension of this work appeared recently in~\cite{CGN}, where the authors studied Kazama--Suzuki type of dualities between subregular $\mathcal{W}$-algebras and principal $\mathcal{W}$-superalgebras.

Assume that $U$, $V$ are vertex (super)algebras. We say that $V$ is the \emph{Kazama--Suzuki dual} of~$U$ if there exist injective homomorphisms of vertex superalgebras
\begin{equation} \label{KSdef}
	\varphi_1 \colon\ V \rightarrow U \otimes \F_1, \qquad \varphi_2 \colon\ U \rightarrow V \otimes \F_{-1},
\end{equation}
so that
$V \cong \operatorname{Com} ( M_{H_1} (1) , U \otimes \F_1)$, $ U \cong \operatorname{Com} ( M_{H_2} (1) , V \otimes \F_{-1})$,
where $M_{H_1} (1) $ (resp.\ $ M_{H_2} (1)$) is a rank one Heisenberg vertex subalgebra of
$U\otimes \F_1$ (resp.\ $V\otimes \F_{-1}$) and $\F_{\pm 1}$ are lattice vertex superalgebras associated to rank one lattice $L= {\Z}\varphi^{\pm} $, $\big\langle \varphi^{\pm}, \varphi^{\pm} \big\rangle = \pm 1$ (cf.\ Section \ref{F_{-1}}).

An immediate benefit of this type of duality is that it provides us with an insight into the representation theory of the dual algebra, i.e., the classification and realization of modules for one algebra can be obtained from its dual. It is therefore particularly useful to find dual pairs of algebras where the representation theory of one of them is better understood, and use it to study the representations of the other one. An important question is how to find and classify instances of this type of duality.

{\bf Classifying Kazama--Suzuki duality between $\boldsymbol{\mathcal{W}_k(\mathfrak{sl}_4, f_{\rm sub})}$ and $\boldsymbol{L^{N=2}_c}$.} In this paper, we classify all possible occurrences of Kazama--Suzuki duality between the $N=2$ superconformal algebra $L^{N=2}_c$ and the subregular algebra $\mathcal{W}$-algebra $\mathcal{W}_{k}(\mathfrak{sl}_4, f_{\rm sub})$.

In general, in order for two algebras to be in Kazama--Suzuki type duality, their Heisenberg cosets need to coincide. Recall that $L^{N=2}_c$ and $L_s(\mathfrak{sl}_2)$ are in Kazama--Suzuki duality for $c =\frac{3s}{s+2}$ (cf.\ \cite{A-IMRN}), hence their coset subalgebras will coincide. That means that a necessary condition for $\mathcal{W}_k(\mathfrak{sl}_4, f_{\rm sub})$ and $L^{N=2}_c$ to be in duality is that the coset algebra $ \mathcal C_k = \operatorname{Com} (M_H(1) ,\mathcal W_k(\mathfrak{sl}_4, f_{\rm sub})) $ of $\mathcal W_k(\mathfrak{sl}_4, f_{\rm sub})$ and the parafermion algebra $ \mathcal N_s(\mathfrak{sl}_2) = \operatorname{Com} (M_H(1) , L_s(\mathfrak{sl}_2)) $ of $L_s(\mathfrak{sl}_2)$ coincide, where $M_H(1)$ denotes the rank one Heisenberg vertex algebra generated by $H$.

To determine at which levels $k$, $s$ the coset algebras $\mathcal C_k$ and $\mathcal N_s(\mathfrak{sl}_2)$ coincide, we rely on the fact that these algebras have a realization as certain quotients of the universal two-parameter vertex algebra $\mathcal W (c, \lambda)$ from \cite{Linshaw}. These quotients are parametrized by certain rational curves, and there is a simple criterion from \cite{Linshaw} which determines when they are isomorphic (cf.\ Section~\ref{univ_W}). This gives us a list of potential candidates for KS duality.

However, it is easy to observe that for $\mathcal{W}_k(\mathfrak{sl}_4, f_{\rm sub})$ and $L^{N=2}_c$ to be in duality, i.e., for there to exist an embedding of type \eqref{KSdef}, \smash{$(G^+)^2 = G^+ _{(-1)} G^+$} needs to be zero in $\mathcal W_k(\mathfrak{sl}_4, f_{\rm sub})$.
 Indeed, $G^+ \otimes e^{\varphi}$ should be proportional to the fermionic generator of the $N=2$ superalgebra, which we denote by $E$, and for which we need to have that $E_{(0)} E =0$. This leads to the conclusion that
 \[
 0 = (G^+ \otimes e^{\varphi} )_{(0)}  (G^+ \otimes e^{\varphi} ) = (G^+) ^2 \otimes e^{2 \varphi},
 \]
 and therefore $(G^+)^2 =0$. In other words, either $G^+$ or $(G^+)^2$ needs to be a singular vector in~$\mathcal W^k(\mathfrak{sl}_4, f_{\rm sub})$.

Applying these two necessary conditions leaves us with two possible candidates: $k=-1$, $c=-15$ and $k=-\frac 73$, $c=1$. We construct explicit Kazama--Suzuki embeddings in those two cases, proving the following.

\begin{Theorem}The vertex algebras $L^{N=2}_c$ and $W_k(\mathfrak{sl}_4, f_{\rm sub})$ are in Kazama--Suzuki duality if and only if $k=-1$ and $c=-15$ or $k=-\frac 73$ and $c=1$.
\end{Theorem}

According to a result of \cite{CGN}, the subregular and principal $\mathcal W$-algebras $\mathcal{W}^{k_1}(\mathfrak{sl}_n, f_{\rm sub})$ and $\mathcal{W}^{k_2} (\mathfrak{sl}(1|n), f_{\rm pr})$ are in KS duality for $k_1, k_2 \in \mathbb{C}$ satisfying the relation
\[ (k_1+ n) (k_2 + n - 1) = 1. \]

For $n=2$, this result recovers the `original' Kazama--Suzuki duality of $L^{N=2}_c$ and $L_s(\mathfrak{sl}_2)$, since $\mathcal{W}^{k'} (\mathfrak{sl}(1 \vert 2), f_{\rm pr})$ is exactly the $N=2$ superconformal algebra of central charge $c= - 3( 1+ 2k')$.

The coset algebras of $\mathcal W_{-n+1} (\mathfrak{sl}_{n+2}, f_{\rm sub})$ and $\mathcal W_{-n+ 1/3} (\mathfrak{sl}_n, f_{\rm sub})$ coincide by a result of \cite{CL} (see also \cite{Linshaw}). Using the duality of \cite{CGN}, it follows that the coset algebras of $\mathcal W_{4-n} (\mathfrak{sl}(1 \vert n), f_{\rm pr})$ and $\mathcal W_{-n+1} (\mathfrak{sl}_{n+2}, f_{\rm sub})$ are also isomorphic.
Since the necessary condition for $\mathcal W_{4-n} (\mathfrak{sl}(1\vert n), f_{\rm pr})$ and $\mathcal W_{-n+1} (\mathfrak{sl}_{n+2}, f_{\rm sub})$ to be in Kazama--Suzuki duality is fulfilled, we conjecture that the same relation holds more generally.

\begin{conj}Let $k_1 = -n+4$ and $k_2=-n+1$. Then there is a Kazama--Suzuki type duality between
$\mathcal W_{k_1} (\mathfrak{sl}(1\vert n), f_{\rm pr})$ and $\mathcal W_{k_2} (\mathfrak{sl}_{n+2}, f_{\rm sub})$.
\end{conj}

Our result above states that $\mathcal W_{k_1} (\mathfrak{sl}(1\vert 2), f_{\rm pr})$ and $\mathcal W_{k_2}(\mathfrak{sl}_4, f_{\rm sub})$ are in Kazama--Suzuki duality for $k_1= 2$, $k_2= - 1$, therefore being a special case of this conjecture for $n=2$.

{\bf Constructing Kazama--Suzuki duality for subregular $\boldsymbol{\mathcal{W}}$-algebras.} Let $\mathfrak{g}$ be a simple finite-dimensional Lie algebra and $k \in \mathbb{C}$. Let $\mathcal{W}^k(\mathfrak{g}, f_{\rm sub})$ be the (universal) subregular $\mathcal{W}$-algebra corresponding to the subregular nilpotent element $f_{\rm sub}$ (see \cite{KW}) and~$\mathcal{W}_k(\mathfrak{g}, f_{\rm sub})$ its simple quotient. The OPEs for subregular $\mathcal{W}$-algebras are not known in general, but for $\mathfrak g = \mathfrak{sl}_n$, it was proved in \cite{Gen} that $\mathcal{W}^k(\mathfrak{g}, f_{\rm sub})$ are isomorphic to the Feigin--Semikhatov algebras \smash{$W_n^{(2)}$} (cf.~\cite{FS}). For $n=2$, $\mathcal{W}^k(\mathfrak{sl}_n, f_{\rm sub})$ is isomorphic to the affine vertex algebra $V^k(\mathfrak{sl}_2)$, while for~${n=3}$ it is the Bershadsky--Polyakov algebra $\mathcal{W}^k(\mathfrak{sl}_3, f_{\rm sub})$.

 In the paper \cite{AK-2022}, we showed that the affine vertex {super}algebra $L_{k'} (\mathfrak{osp}(1 \vert 2))$ is the Kazama--Suzuki dual of the Bershadsky--Polyakov algebra $\mathcal{W}_k(\mathfrak{sl}_3, f_{\rm sub})$ for $k'=-\frac 54$ and $k=1$.
 More precisely, there exists an embedding of $\mathcal{W}_k(\mathfrak{sl}_3, f_{\rm sub})$ into the tensor product of the affine vertex {super}algebra $L_{k'} (\mathfrak{osp}(1 \vert 2))$ at level $k'=-\frac 54$ and the Clifford vertex {super}algebra $F$ (which is, by the boson-fermion correspondence, isomorphic to the lattice vertex superalgebra $\F_1$); and the corresponding inverse embedding.
 In \cite{A-2019}, it was proved that \smash{$L_{-5/4}(\mathfrak{osp}(1\vert 2))$} can be realized as the vertex {super}algebra $F^{1/2} \otimes \Pi^{1/2}(0)$, where $\Pi^{1/2}(0)$ is a lattice type vertex algebra, and~$F^{1/2}$ is a Clifford vertex {super}algebra.
 Using this result, and the fact that at level~${k'=-\frac 54}$ there is a conformal embedding of $L_{k'}(\mathfrak{sl}_2)$ into $L_{k'}(\mathfrak{osp}(1 \vert 2))$ (cf.\ \cite{AKMPP-2020}),
we obtained a~realization of the Bershadsky--Polyakov algebra $\mathcal{W}_k(\mathfrak{sl}_3, f_{\rm sub})$ and an explicit construction of irreducible $\mathcal{W}_k(\mathfrak{sl}_3, f_{\rm sub})$-modules.
Relaxed modules for $L_{k'}(\mathfrak{osp}(1 \vert 2))$ are mapped to the ordinary~${\mathcal{W}_k(\mathfrak{sl}_3, f_{\rm sub})}$-modules, for which one expects it is easier to obtain the tensor category structure and calculate the fusion rules.

 In this paper, we construct explicit Kazama--Suzuki type embeddings between the subregular $\mathcal{W}$-algebra $\mathcal{W}_k(\mathfrak{sl}_4, f_{\rm sub})$ and the $N = 2$ superconformal algebra $L^{N=2}_{c}$ for $k=-1$, $c=-15$, thus obtaining a realization of $\mathcal{W}_k(\mathfrak{sl}_4, f_{\rm sub})$ and its irreducible modules.

 We show (cf.\ Propositions \ref{ulaganje-1} and \ref{ulaganje-2}) the following.

 \begin{Theorem}
 	There exist embeddings of vertex superalgebras
 \[
 \Phi\colon\ L^{N=2}_{c=-15} \rightarrow \mathcal{W}_{-1}(\mathfrak{sl}_4, f_{\rm sub}) \otimes \F_{-1}, \qquad \Phi^{\rm inv} \colon\ \mathcal{W}_{-1}(\mathfrak{sl}_4, f_{\rm sub}) \rightarrow L^{N=2}_{c=-15} \otimes \F_{1}
 \]
 such that
 \begin{gather*}
 	L^{N=2}_{c=-15} \cong \operatorname{Com} (M_{H^{\perp}} (1), \mathcal{W}_{-1}(\mathfrak{sl}_4, f_{\rm sub}) \otimes \F_{-1}), \\
 	\mathcal{W}_{-1}(\mathfrak{sl}_4, f_{\rm sub}) \cong \operatorname{Com} \bigl(M_{\overline{H}}(1), L^{N=2}_{c=-15} \otimes \F_{1}\bigr),
 \end{gather*}
 where $M_{H^{\perp}}(1)$ and \smash{$M_{\overline{H}} (1)$} are Heisenberg vertex algebras.
 \end{Theorem}

{\bf Classification of irreducible $\boldsymbol{\mathcal{W}_k(\mathfrak{sl}(4), f_{\rm sub})}$-modules.} As a consequence of the duality, we are able to construct a realization of irreducible $\mathcal{W}_{k=-1}(\mathfrak{sl}_4,\allowbreak f_{\rm sub})$-modules as modules for~${L^{N=2}_c \otimes \F_{1}}$. Furthermore, the highest weight $\mathcal{W}_k(\mathfrak{sl}_4, f_{\rm sub})$-modules are parameterized by zeroes of certain curves (cf.\ Theorem~\ref{klas-1}).

 Let $L(x,y,z)$ be the irreducible highest weight $\mathcal{W}_k(\mathfrak{sl}_4, f_{\rm sub})$-module, generated by the highest weight vector $v_{x,y,z}$.
An important property here is that \smash{$(G^+)^2$} is a singular vector in~$\mathcal{W}^{-1}(\mathfrak{sl}_4,\allowbreak f_{\rm sub})$ (cf.\ Lemma \ref{sing_G^+}), hence the top space $L(x,y,z)_{\text{\rm top}}$ can be either 1- or 2-dimensional.
 Let
\begin{gather*}
 S_{1} = \left\{ \left(-\frac{h}{4}, q+ \frac{h^2}{8}, -\frac{1}{25} (h+5) \bigl(h^2 - 5h + 15 q^2\bigr)\right),\, (h,q) \in {\C} ^2\right\}, \\
 S_{2} = \left\{ \left(-\frac{h{-}5}{4}, q+ \frac{h^2 - 2 h + 5}{8}, -\frac{1}{25} (h+5) \bigl(h^2 - 5h + 15 q^2\bigr)\right),\, (h,q) \in {\C} ^2\right\} .
\end{gather*}

We prove the following statement.

\begin{Theorem}
The set $\{L(x,y,z), (x,y,z) \in S_1 \cup S_2 \}$ provides a complete list of irreducible, highest weight $\mathcal{W}_{k=-1}(\mathfrak{sl}_4, f_{\rm sub})$-modules. Moreover,
$
 \dim L(x,y,z)_{\rm top} = 1 \iff (x,y,z) \in S_1$.
\end{Theorem}

{\bf Setup.}
\begin{itemize}\itemsep=0pt
	\item We adopt the following notation for the vertex operator corresponding to the state $a$
\[
Y (a, z) = \sum _{i \in \mathbb{Z}} a_{(i)}z^{-i-1}.
\]
\item The Zhu algebra associated to the vertex operator algebra $V$ with the Virasoro vector $\omega$ will be denoted with $A_{\omega}(V)$.
\item Let $\F_{\pm 1}$ be the lattice vertex {super}algebras associated to the lattice ${\mathbb Z}\sqrt{\pm 1}$ defined in Section \ref{F_{-1}}.
\item Let $M_{\alpha}(1)$ denote the rank one Heisenberg vertex algebra generated by a Heisenberg vector~$\alpha$.
\item Let $\mathcal{W}^k(\mathfrak{sl}_n, f_{\rm sub})$ be the (universal) subregular $\mathcal{W}$-algebra corresponding to the subregular nilpotent element $f_{\rm sub}$ and $\mathfrak g = \mathfrak{sl}_n$, and $\mathcal{W}_k(\mathfrak{sl}_n, f_{\rm sub})$ its simple quotient.
\item $L_s(\mathfrak{sl}_2)$ is affine Lie algebra of $\mathfrak{sl}_2$ and $N_s(\mathfrak{sl}_2)$ its parafermion (coset) subalgebra.
\item The universal $N=2$ superconformal vertex algebra of central charge $c$ is denoted by $V^{N=2}_c$ and $L^{N=2}_{c}$ is its simple quotient.
\end{itemize}

\section{Preliminaries}	

\subsection[Lattice vertex superalgebras F\_1]{ Lattice vertex superalgebras $\boldsymbol{\mathcal F_{\pm 1}}$} \label{F_{-1}}

Consider a rank one lattice $L= {\Z}\varphi^{\pm} $, $\big\langle \varphi^{\pm}, \varphi^{\pm} \big\rangle = \pm 1$. Let $\F_{1}$ (resp.\ $\F_{-1}$) be the associated vertex algebra. These vertex superalgebras are used for the construction of the inverse of Kazama--Suzuki functor in the context of duality between affine $\widehat{sl}_2$ and $N=2$ superconformal algebra (cf.\ \cite{A-IMRN,A-2001,FST}).

As a vector space $\F_{\pm 1} = {\C}[L] \otimes M_{\varphi^{\pm}} (1)$, where ${\C}[L]$ is a group algebra of $L$, and $M_{\varphi^{\pm}} (1)$ the Heisenberg vertex algebra generated by the Heisenberg field $\varphi^{\pm}(z) = \sum_{n \in {\Z} } \varphi^{\pm}(n) z^{-n-1}$ such that
\[
 \bigl[\varphi^{\pm}(n), \varphi^{\pm}(m)\bigr] = \pm n \delta_{n+m, 0}.
 \]

The vertex algebras $\F_{\pm 1}$ are weakly generated by \smash{$e^{ \varphi^{\pm}}$} and \smash{$e^{-\varphi^{\pm}}$}.
Moreover, $\F_{\pm 1}$ is a simple vertex superalgebra and a completely reducible $M_{\varphi^{\pm}} (1)$-module isomorphic to
\[
\F_{\pm 1} \cong \bigoplus_{m \in \mathbb{Z}} \F_{\pm 1}^m ,
\]
where $ \F_{\pm 1}^m$ is an irreducible $M_{\varphi^{\pm}} (1)$-module generated by \smash{$e^{m\varphi^{\pm}}$}.

\subsection[N=2 superconformal algebra]{$\boldsymbol{ N=2}$ superconformal algebra}

The universal $N=2$ superconformal vertex algebra $V^{N=2}_{c}$ is generated by even fields $T(z)$, $H(z)$ and odd fields $E(z)$, $F(z)$
satisfying the following OPEs
\begin{gather*}
 T(z)T(w) \sim \frac{\frac{c}{2}}{(z-w)^{4}}+\frac{2T(w)}{(z-w)^{2}} + \frac{\partial T(w)}{z-w}, \qquad H(x)H(y) \sim \frac{ \frac{c}{3}}{(z-w)^{2}}, \qquad \\
 T(z)H(w) \sim \frac{H(w)}{ (z-w)^{2}}+ \frac{\partial H(w)}{z-w}, \\
 T(z)E(w) \sim \frac{\frac{3}{2} E(w)}{(z-w)^{2}} + \frac{\partial E(w)}{z-w}, \qquad T(z)F(w) \sim \frac{\frac{3}{2} F(w)}{(z-w)^{2}} + \frac{\partial F(w)}{z-w}, \\
	E(z)F(w) \sim \frac{\frac{2c}{3}}{(z-w)^{3}}+\frac{2H(w)}{(z-w)^{2}} + \frac{2T(w) + \partial H(w)}{z-w}, \\
	F(z)E(w) \sim \frac{\frac{2c}{3}}{(z-w)^{3}}-\frac{2H(w)}{(z-w)^{2}} + \frac{2T(w) - \partial H(w)}{z-w}, \\
	H(z)E(w) \sim \frac{E(w)}{z-w}, \qquad H(z)F(w) \sim -\frac{F(w)}{z-w}, \qquad E(z)E(w) \sim 0, \qquad F(z)F(w) \sim 0. 	
\end{gather*}

Set
\begin{gather*}
 T(z) = \sum _{i \in {\Z}} T(i) z^{-i-2}, \qquad H(z) = \sum _{i \in {\Z}} H(i) z^{-i-1},\\
 E(z) = \sum _{i \in {\Z}} E_{(i)} z^{-i-1} = \sum _{i \in {\Z}} E\left(i+\frac{1}{2}\right)z^{-i-2},\\
 F(z) = \sum _{i \in {\Z}} F_{(i)} z^{-i-1} = \sum _{i \in {\Z}} F\left(i+\frac{1}{2}\right)z^{-i-2}.
 \end{gather*}
 Then the components of these fields satisfy the commutation relation for the $N=2$ superconformal algebra with basis $\{T(n), H(n), E(r), F(r)\}$, $n \in \mathbb{Z}$, $r \in \frac{1}{2}+\mathbb{Z}$,
\begin{gather*}
	[T(m), T(n)] = (m-n)T(m+n) + \frac{c}{12}\bigl(m^3 - m\bigr)\delta_{m+n,0}, \qquad [H(m), H(n)] = \frac{c}{3}m\delta_{m+n,0}, \\
	[T(m), H(n)] = -nH(n+m), \qquad [ T(m) , E(r) ] = \left(\frac{1}{2} m-r \right)E(m+r), \\
 [ T(m) , F(r) ] = \left(\frac{1}{2} m-r \right)F(m+r), \\
	[ H(m) ,E(r)] = E(m+r), \qquad [ H(m) , F(r)] = -F(m+r), \\
	\{E(r), F(s) \} = 2 T({r + s}) + ( r - s )H(r + s) + \frac{c}{3}\left( r^2 - \frac{1}{4} \right)\delta_{r+s,0}, \\
	\{E(r), E(s) \} = \{F(r), F(s) \} = 0.
\end{gather*}

 Let $L^{N=2}_{c}$ be the simple quotient of $V^{N=2}_{c}$.

\subsubsection{Spectral flow automorphisms}
The $N=2$ superconformal algebra $V^{N=2}_{c}$ admits a family of spectral flow automorphisms $\sigma^\ell$, $\ell \in \mathbb{Z}$, given by
\begin{align*}
\begin{aligned}
	& \sigma^\ell(T(n))= T(n) -\ell H(n) + \frac{1}{6}\ell ^2\delta_{n,0}c\mathbbm{1}, \qquad \sigma^l(H(n)) = H(n) + \frac{1}{3}\ell \delta_{n,0}c\mathbbm{1},  \\
	& \sigma^\ell(E(r))= E({r-\ell}), \qquad \sigma^\ell(F(r)) = F({r+\ell}), \qquad \sigma^\ell(\mathbbm{1}) = \mathbbm{1}.
\end{aligned}
\end{align*}

Define
\[
 \Delta (-H,z) = z^{-H(0)}\exp \left ( \sum _{k=1} ^{\infty} (-1)^{k+1}\frac{-H(k) }{kz^k}\right ).
 \]

For any $V^{N=2}_{c}$-module $(M, Y_M(\cdot, z))$ and $n \in \mathbb{Z}$,
$\sigma^ n (M)$ is again a $L^{N=2}_{c}$-module with vertex operator structure given by
$ Y_{\sigma ^n (M) } ( \cdot , z)) := Y_{ M } (\Delta ( -n H, z) \cdot , z))$.

\subsubsection{Twisted highest weight conditions}
Let us define a new Virasoro vector $\widetilde T = T + \partial H$. Then $E(z)$ is primary field of conformal weight~$1/2$ and $F(z)$ of conformal weight $5/2$ with respect to $\widetilde T$.
 With respect to the new grading operator $\widetilde T(0)$ we shall also need to introduce twisted highest weight conditions.
 Let~${(h,q) \in {\C}^2}$. Let $M_c[h,q]$ be the Verma module of the $N=2$ superconformal algebra of central charge $c$ generated by the (twisted) highest weight vector $v_{h,q}$ such that for $n \in {\Z}_{\ge 0}$ (cf.\ \cite{FST}):
\begin{gather*}
 H(n) v_{h,q} =\delta_{n, 0} h v_{h,q}, \qquad T(n) v_{h,q} = \delta_{n,0} q v_{h,q}, \\ E\left(n-\frac{1}{2}\right) v_{h,q} = F\left(n+\frac{3}{2}\right) v_{h,q} =0.
 \end{gather*}
 Let $L_c[h,q]$ be the irreducible quotient of $M_c[h,q]$. Then $L_c[h,q]$ is an irreducible $V_c^{N=2}$-module.

\subsubsection{Parafermionic subalgebras}
\label{paraferm-1}
Let $ \mathcal N_c = \operatorname{Com} \bigl(M_H(1) , L^{N=2}_{c}\bigr) $ be the parafermionic subalgebra of $L^{N=2}_{c}$. Note that using the Kazama--Suzuki duality between $L^{N=2} _c$ and $L_k(\mathfrak{sl}_2)$ (cf.\ \cite{A-IMRN, FST}), we have that
$\mathcal N_c = \mathcal N_s (\mathfrak{sl}_2)$, where $c =\frac{3s}{s+2}$ and
$\mathcal N_s (\mathfrak{sl}_2)$ is the parafermion vertex subalgebra of $L_s(\mathfrak{sl}_2)$ (cf.\ \cite{DLY}). If $c \notin \{ 0, 1, \frac{3}{2}, -6, -9\}$, then the vertex algebra $ \mathcal N_c $ is (weakly) generated by the fields $ T^{\perp}$ and~$W^{N=2}_{c}$ (cf.\ \cite{BEHHH,DLY}), where
\begin{gather*}
	 T^{\perp} = T - \frac{3}{2c}{:}HH{:}, \qquad
	W^{N=2}_{c} = \nu \left({:} E F{:} - \partial T -\frac{6}{c}{:} T H{:} - \frac{c-9}{3c}\partial^2H + \frac{6}{c^2}{:}H^3{:} \right),
\end{gather*}
and $\nu \in \mathbb{C}$ is a normalization factor. In particular, $\mathcal N_{c=-15}$ is weakly generated by $T^{\perp}$ and $W^{N=2}_{c=-15}$.

 For other central charges, we have the following:
\begin{itemize}\itemsep=0pt
\item If $c\in \{ 0, 1\}$, then $\mathcal N_c = {\C} {\bf 1}$ (cf.\ \cite{DLY}).
\item $N_{c=3/2}$ is the simple Virasoro vertex algebra of central charge $c=\frac{1}{2}$.
 \item \smash{$\mathcal N_{c=-6}=\mathcal N_{-4/3}(\mathfrak{sl}_2)$} is the singlet vertex algebra $\mathcal M(3)$ generated by the Virasoro field and primary field of conformal weight $5$ (cf.\ \cite{A-2005}).
 \item \smash{$\mathcal N_{c=-9} =\mathcal N_{-3/2} (\mathfrak{sl}_2)$} is a direct sum of \smash{$\mathcal W_{-5/2} (\mathfrak{sl}_3, f_{\rm pr})$}-modules (cf.\ \cite{AMW}). As a vertex algebra it is generated by \smash{$\mathcal W_{-5/2} (\mathfrak{sl}_3, f_{\rm pr})$} and its simple module of conformal weight $4$.
\end{itemize}

 Let $A_{\omega}(V)$ be the Zhu algebra associated to the VOA $V$ with the Virasoro vector $\omega$, and let~$[v]$ be the image of $v \in V $ under the mapping $V \mapsto A_{\omega}(V)$.

 One can show that the Zhu's algebra $A\bigl(V^{N=2}_{c}\bigr)$ (cf.\ \cite{A-IMRN}) is isomorphic to ${\mathbb C}[x,y]$; i.e., there is isomorphism $f \colon A\bigl(V^{N=2}_{c}\bigr) \rightarrow {\mathbb C}[x_1,y_1]$ such that
$[H({-1}){\mathbbm 1}] \mapsto x_1, \ [T({-2}{)\mathbbm 1}] \mapsto y_1$.

\subsection{Kazama--Suzuki duality}
In this subsection, we will define a duality of vertex algebras which is motivated by the duality between $N=2$ superconformal vertex algebra and affine vertex algebra $L_k(\mathfrak{sl}_2)$.

Recall first that if $S$ is a vertex subalgebra of $V$, we have the commutant subalgebra of $V$ (cf.\ \cite{LL})
\[
\operatorname{Com} (S, V):= \{ v \in V \mid a_n v = 0,\, \forall a \in S,\, \forall n \in {\Z}_{\ge 0}\}.
\]

Assume that $U$, $V$ are vertex superalgebras. We say that $V$ is the Kazama--Suzuki dual of~$U$ if there exist injective homomorphisms of vertex superalgebras
$
 \varphi_1 \colon V \rightarrow U \otimes \F_1$, $ \varphi_2 \colon U \rightarrow V \otimes \F_{-1}$,
so that
$
V \cong \operatorname{Com} \left( M_{H_1} (1), U \otimes \F_1\right)$, $ U \cong \operatorname{Com} \left( M_{H_2} (1), V \otimes \F_{-1}\right)$,
where $M_{H_1} (1) $ (resp.\ $ M_{H_2} (1)$) is a rank one Heisenberg vertex subalgebra of
$U\otimes \F_1$ (resp.\ $V\otimes \F_{-1}$), and $\F_{\pm 1}$ are lattice vertex superalgebras defined in Section~\ref{F_{-1}}.

\subsection[Subregular W-algebras W\^{}k(sl\_4, f\_sub]{Subregular $\boldsymbol{\mathcal{W}}$-algebras $\boldsymbol{\mathcal{W}^k(\mathfrak{sl}_4, f_{\rm sub})}$}	\label{subreg_ope}

Let $\mathfrak{g}$ be a simple finite-dimensional Lie algebra and $k \in \mathbb{C}$. Let $\mathcal{W}^k(\mathfrak{g}, f_{\rm sub})$ be the (universal) subregular $\mathcal{W}$-algebra corresponding to the subregular nilpotent element $f_{\rm sub}$ (\cite{KW}).

In the case where $\mathfrak g = \mathfrak{sl}_n$, it was proved in \cite{Gen} that $\mathcal{W}^k(\mathfrak{g}, f_{\rm sub})$ is isomorphic to the Feigin--Semikhatov algebra \smash{$W_n^{(2)}$} (cf.\ \cite{FS}), using a certain free field realization of $\mathcal{W}^k(\mathfrak{g}, f_{\rm sub})$.

In this paper, we will only consider the case $\mathfrak g = \mathfrak{sl}_4$. The vertex algebra $\mathcal{W}^{k}(\mathfrak{sl}_4, f_{\rm sub})$ has central charge
\[
c_k = -\frac{(3k+8)(8k+17)}{(k+4) },
\]
 and it is freely generated by fields $J(z)$, $L(z)$, $G^+ (z)$, $G^- (z)$, $W(z)$. The explicit OPEs are known (see \cite{CL18,FS}) and are written in Appendix~\ref{appendix-b}. In this paper, we shall assume that the gradation on~$\mathcal{W}^{k}(\mathfrak{sl}_4, f_{\rm sub})$ is defined by the shifted Virasoro field $\widetilde L = L + \partial J$. Then $G^+$ (resp.~$G^-$) is a~primary field of conformal weight $1$ (resp.~$3$). We set
\begin{gather*}
J(z) = \sum _{n \in {\Z}} J(n) z^{-n-1}, \qquad L(z) = \sum _{n \in {\Z}} L(n) z^{-n-2}, \qquad G^+ (z) = \sum _{n \in {\Z}} G^+(n) z^{-n-1}, \\ G^- (z) = \sum _{n \in {\Z}} G^-(n) z^{-n-3}, \qquad W (z) = \sum _{n \in {\Z}} W(n) z^{-n-3}.
\end{gather*}	
The Zhu algebra $A\bigl(\mathcal{W}^{k}(\mathfrak{sl}_4, f_{\rm sub})\bigr)$ is generated by
\[
 E=[G^+],\qquad F=[G^- ], \qquad X=[J ], \qquad Y =[L], \qquad Z=[W].
 \]

The subregular $\mathcal{W}$-algebra $\mathcal{W}^k(\mathfrak{sl}_4, f_{\rm sub})$ admits a family of spectral flow automorphisms $\psi^m$, $m \in \mathbb{Z}$, given by
\begin{gather*}
 \psi^m (J(n)) = J(n)- { \frac{m (3 k +8)}{4} }
	 \delta _{n,0}\mathbbm{1} , \qquad \psi^m ( L(n)) = L(n) - mJ(n) + \frac{m^2 { (3k+8)}}{{8} }\delta _{n,0}\mathbbm{1}, \\
 \psi^m (G^+(n)) = G^+(n-m), \qquad \psi^m (G^-(n)) = G^-(n+m).
 \end{gather*}

Define
 \[
 \Delta (-J,z) = z^{-J(0)}\exp \left ( \sum _{k=1} ^{\infty} (-1)^{k+1}\frac{-J(k)}{kz^k}\right ).
 \]

For any $\mathcal{W}^k(\mathfrak{sl}_4, f_{\rm sub})$-module $(M, Y_M(\cdot, z))$ and $n \in \mathbb{Z}$,
$\psi^m (M)$ is again a $\mathcal{W}^k(\mathfrak{sl}_4, f_{\rm sub})$-module with vertex operator structure given by
$ Y_{\psi^m (M) } ( \cdot , z)) := Y_{ M } (\Delta ( -m J, z) \cdot , z))$.

\subsection[Singular vectors in W\^{}k(sl\_n, f\_sub)]{Singular vectors in $\boldsymbol{\mathcal{W}^k(\mathfrak{sl}_n, f_{\rm sub})}$}
	
	Formulas for singular vectors in $\mathcal{W}^k(\mathfrak{sl}_n, f_{\rm sub})$ are not known in general; however, the following criterion from \cite{Feh} tells us when $(G^+)^s$ is singular.
	
	\begin{Proposition}[\cite{Feh}]\label{kriterij}
		The vector $(G^+)^s$, $s>0$, is singular in $\mathcal{W}^k(\mathfrak{sl}_n, f_{\rm sub})$ if and only if~${
		 	i(k+n-1)=s}
$
 for some $i \in \{1,\dots ,n-1 \}$.
	\end{Proposition}

Using the above criterion, it is easy to observe the following.
	\begin{Lemma} \label{sing_G^+}\qquad
	\begin{itemize}\itemsep=0pt
		\item[$(i)$] For $k=-n+3$, $(G^+)^2$ is singular in $\mathcal{W}^k(\mathfrak{sl}_n, f_{\rm sub})$ but $G^+$ is not.
		\item[$(ii)$] For $k=-n+4$, $(G^+)^3$ is singular in $\mathcal{W}^k(\mathfrak{sl}_n, f_{\rm sub})$ but $(G^+)^2$ is not.	
	\end{itemize}
	\end{Lemma}
	
	\begin{proof}
		(i) Assume that $(G^+)^2$ is singular in $\mathcal{W}^k(\mathfrak{sl}_n), f_{\rm sub})$ but $G^+$ is not. Then from Proposition \ref{kriterij} it follows that there exists $i \in \{1,\dots ,n-1 \}$ such that $ i(k+n-1)=2 $ and there does not exist $i \in \{1,\dots ,n-1 \}$ such that $ i(k+n-1)=1$. Let now $k=-n+3$. Then $i=1$ is solution of the first equation, and the second equation does not have solutions. This proves (i).
		Analogous reasoning yields (ii).
	\end{proof}
	
\subsection[Highest weight W\^{}k(sl\_4, f\_sub)-modules]{Highest weight $\boldsymbol{\mathcal W^k(\mathfrak{sl}_4, f_{\rm sub})}$-modules}

\begin{Definition} Let $(x,y,z) \in \mathbb{C}^3$. We say that a module $M$ is a highest weight $ \mathcal{W}^k(\mathfrak{sl}_4, f_{\rm sub})$-module of highest weight $(x,y,z)$ if there exist a highest weight vector $v_{x,y,z}$ such that
\begin{gather*}
 J(0)v_{x,y,z} = x v_{x,y,z}, \qquad J(n)v_{x,y,z} = 0 \qquad \text{for} \quad n>0, \\
 L(0)v_{x,y,z} = y v_{x,y,z}, \qquad L(n)v_{x,y,z} = 0 \qquad \text{for} \quad n>0, \\
 W(0)v_{x,y,z} = z v_{x,y,z}, \qquad W(n)v_{x,y,z} = 0 \qquad \text{for}\quad n>0, \\
 G^{+}(n-1)v_{x,y,z} = G^{-}(n)v_{x,y,z} =0 \qquad \text{for} \quad n\geq 1,
\end{gather*}
and $M= \mathcal{W}^k(\mathfrak{sl}_4, f_{\rm sub}). v_{x,y,z}$.
\end{Definition}

 If an irreducible highest weight module of highest weight $(x,y,z)$ exists, it is unique up to equivalence, and we denote it by $L(x,y,z)$.
 It is ${\Z}_{\ge 0}$ graded with respect to $\widetilde L(0)$ and its top space
 \[
 L(x,y,z)_{\text{\rm top}} = \big\lbrace v \in L(x,y,z) \mid \widetilde L(0)v = \widetilde y v \big\rbrace \]
 is spanned by $ \big\{G^-(0)^iv_{x,y,z}, \, i \geq 0 \big\}$.

 \begin{Remark} Note that $\widetilde L(0) \vert_{ L(x,y,z)_{\text{\rm top}}} \equiv \widetilde y \operatorname{Id}$, where $\widetilde y = y-x$. But since all OPE formulas in Appendix~\ref{appendix-b} are expressed in terms of the original Virasoro field $L$ we choose to express all highest weight with respect to
 $(J(0), L(0), W(0))$. One could also easily rewrite highest weights with respect to $\bigl(J(0), \widetilde L(0), W(0)\bigr)$.
\end{Remark}

\begin{Remark} One can show that for each $(x,y,z) \in \mathbb{C}^3$, there exist an irreducible highest weight $ \mathcal{W}^k(\mathfrak{sl}_4, f_{\rm sub})$-module with these highest weights using the results of~\cite{KW} and/or inverse reduction construction from \cite{Feh}. But in the current paper we do not need this general result. In what follows, we shall construct highest weight modules in some special cases which are relevant to KS duality.
\end{Remark}

\section[Classification of Kazama--Suzuki duality between the N=2 superconformal algebra L\^{}N=2\_c and the subregular W-algebra W\_k(sl\_4, f\_sub]{Classification of Kazama--Suzuki duality between\\ the $\boldsymbol{N=2}$ superconformal algebra $\boldsymbol{L^{N=2}_c}$\\ and the subregular $\boldsymbol{\mathcal{W}}$-algebra $\boldsymbol{\mathcal{W}_{k}(\mathfrak{sl}_4, f_{\rm sub})}$} \label{KS_gen}

In this section, we classify all possible occurrences of Kazama--Suzuki duality between the $N=2$ superconformal algebra $L^{N=2}_c$ and the subregular algebra $\mathcal{W}$-algebra $\mathcal{W}_{k}(\mathfrak{sl}_4, f_{\rm sub})$.

In order for those two algebras to be in KS duality, two necessary conditions need to be fulfilled: (i) coset subalgebras need to coincide and (ii) $(G^+)^2$ needs to be a singular vector in~$\mathcal W^k(\mathfrak{sl}_4, f_{\rm sub})$.

To check the first condition, we use a criterion from \cite{Linshaw}, which determines when certain quotients of the universal two-parameter vertex algebra $\mathcal W (c, \lambda)$ coincide (cf.\ Section \ref{univ_W}). For the second criterion, we use Lemma \ref{sing_G^+}.

\subsection[Universal two-parameter vertex algebra W (c, lambda)]{Universal two-parameter vertex algebra $\boldsymbol{\mathcal W (c, \lambda)}$} \label{univ_W}
 The universal two-parameter vertex algebra $\mathcal W (c, \lambda)$ was constructed in \cite{Linshaw}. It is defined over the ring $\mathbb{C}[c, \lambda]$ and freely generated by Virasoro field $L$ of central charge $c$ and strong generators~${\big\{W^i \mid i \geq 3\big\}}$ of weight $i \in \mathbb{Z}$. $\mathcal W (c, \lambda)$ is a simple vertex algebra over $\mathbb{C}[c, \lambda]$.

 Let $I \subseteq \mathbb{C}[c, \lambda]$ be an ideal. Then we have the quotient vertex algebra
 \[
 \mathcal W ^{I}(c, \lambda) = \mathcal W (c, \lambda)/ I \cdot \mathcal W (c, \lambda).
 \]
 The variety $V(I)\subseteq \mathbb{C}^2 $ is called a truncation curve for $\mathcal W ^{I}(c, \lambda)$.

 Many important examples of $\mathcal W$-algebras are quotients of (localizations of) the algebra $\mathcal W (c, \lambda)$, including the principal $\mathcal W$-algebra $\mathcal W^k(\mathfrak{sl}_n, f_{\rm prin})$, the parafermion subalgebra
 \[
 \mathcal N_s(\mathfrak{sl}_2) = \operatorname{Com} (M_H(1) , L_s(\mathfrak{sl}_2))
 \]
 of $L_s(\mathfrak{sl}_2)$ and the coset algebra $ \mathcal C_k = \operatorname{Com} (M_H(1) ,\mathcal W_k(\mathfrak{sl}_4, f_{\rm sub})) $ of $\mathcal W_k(\mathfrak{sl}_4, f_{\rm sub})$.

 If $I$ is a maximal ideal of the form $I = (c-c_0, \lambda - \lambda_0)$, for some $c_0, \lambda_0 \in \mathbb{C}$, then $\mathcal W ^{I}(c, \lambda)$ and its simple quotient $\mathcal W_{I}(c, \lambda)$ are vertex algebras over $\mathbb{C}$. Given maximal ideals $I_0 = (c-c_0, \lambda - \lambda_0)$ and $I_1 = (c-c_1, \lambda - \lambda_1)$, let $\mathcal W_0$ and $\mathcal W_1$ be the simple quotients of $\mathcal W^{I_0} (c, \lambda)$ and $\mathcal W^{I_1} (c, \lambda)$.
 The following criterion describes when they are isomorphic.

\begin{Proposition}[{\cite[Corollary 10.1]{Linshaw}}]
	Let $\mathcal W_0$ and $\mathcal W_{1}$ be the simple quotients of $\mathcal W^{I_0} (c, \lambda)$ and~$\mathcal W^{I_1} (c, \lambda)$. Then if $c \neq 0,-2$, $\mathcal W_0$ and $\mathcal W_{1}$ are isomorphic only if $c_0 = c_1$ and $\lambda_0 = \lambda_1$, hence any pointwise coincidences between the simple quotients of $\mathcal W^{I_0} (c, \lambda)$ and $\mathcal W^{I_1} (c, \lambda)$ must correspond to intersection points of the truncation curves $V(I_0) \cap V(I_1)$.
\end{Proposition}

\subsection[Coincidences between parafermion algebras of L\_s(sl\_2) and W\_k(sl\_4, f\_sub) and applications]{Coincidences between parafermion algebras of $\boldsymbol{ L_s(\mathfrak{sl}_2)}$ \\
and $\boldsymbol{\mathcal W_k(\mathfrak{sl}_4, f_{\rm sub})}$ and applications}

Since $L^{N=2}_c$ and $L_s(\mathfrak{sl}_2)$ are in Kazama--Suzuki duality for $c =\frac{3s}{s+2}$ (cf.\ \cite{A-IMRN}), we know that their coset subalgebras will coincide. Therefore, we can use the above criterion in order to determine when the coset algebras $ \mathcal N_s(\mathfrak{sl}_2)$ and $\mathcal C_k$ are isomorphic.

\begin{Proposition} \label{pf_coincidence}
	The parafermionic algebras $ \mathcal N_s(\mathfrak{sl}_2)$ and $\mathcal C_k$ coincide only in the following cases:
	\begin{enumerate}\itemsep=0pt
		\item[$(1)$] $k=-1$ and $s=-\frac 53$;
		\item[$(2)$] $k=-\frac 32$ and $s=-\frac 75$;
		\item[$(3)$] $k=-\frac{13}4$ and $s=-\frac 74$;	
		\item[$(4)$] $k=-\frac 83$ and $s= 0$;	
		\item[$(5)$] $k=-\frac 52$ and $s=1$;	
		\item[$(6)$] $k=-\frac 73$ and $s=1$;	
		\item[$(7)$] $k=-\frac{11}4$ and $s=-\frac 12$.
	\end{enumerate}
\end{Proposition}
\begin{proof}
	By \cite[Corollary 10.1]{Linshaw}, aside from the cases
$c=0,-2$,
all other isomorphisms of simple quotients $\mathcal N_s(\mathfrak{sl}_2) \cong \mathcal C_k $ correspond to intersection points
 of the corresponding truncation curves parameterizing $\mathcal N_s(\mathfrak{sl}_2)$ and $\mathcal C_k $.

The parafermionic subalgebra $\mathcal N_s(\mathfrak{sl}_2) = \operatorname{Com} (M_H(1) , L_s(\mathfrak{sl}_2)) $ is obtained as a simple quotient of $\mathcal W (c, \lambda)$ by setting (cf.\ \cite[Theorem 7.1]{Linshaw})
\begin{equation} \label{N_s}
	c = \frac{2(s - 1)}{s + 2}, \qquad \lambda = \frac{s + 1}{(s -2)(3s + 4)},
\end{equation}
while $ \mathcal C_k = \operatorname{Com} (M_H(1), \mathcal W_k(\mathfrak{sl}_4, f_{\rm sub})) $ of $\mathcal W_k(\mathfrak{sl}_4, f_{\rm sub})$ is obtained by setting (cf.\ \cite[Theorem~7.3]{Linshaw})
\begin{equation} \label{C_k}
	c = -\frac{4(5 + 2k)(7 + 3k)}{4 + k}, \qquad \lambda = - \frac{(3 + k)(4 + k)}{3(2 + k)^{2}(16 + 5k)}.
\end{equation}

 Equating \eqref{N_s} and \eqref{C_k} shows that there are exactly five such points $(k, s)$, namely,
 \[
 \left( -\frac{13}{4}, -\frac{7}{4}\right),\qquad \left( -\frac{8}{3}, 0\right),\qquad \left( -\frac{5}{2}, 1\right),\qquad \left( -\frac{3}{2}, -\frac{7}{5}\right),\qquad \left( -1, -\frac{5}{3}\right).
 \]

 It remains to check the cases in which the parafermionic subalgebras are not obtained as quotients of the universal $\mathcal W$-algebra $\mathcal W (c, \lambda)$, that is, for $c=0$ and $c=-2$.

 If $c=-2$, we have $(k,s) = \bigl(-2, -\frac{1}{2}\bigr)$ or $(k,s) = \bigl(-\frac{11}{4}, -\frac{1}{2}\bigr)$.
 If $(k,s) = \bigl(-\frac{11}{4}, -\frac{1}{2}\bigr)$, then the parafermionic subalgebras $\mathcal N_s(\mathfrak{sl}_2)$ and $\mathcal C_k $ are both isomorphic to the singlet algebra $\mathcal M(2)$ (cf.~\cite{Ridout,Wang}). If $(k,s) = \bigl(-2, -\frac{1}{2}\bigr)$, then $\mathcal W_{-2}(\mathfrak{sl}_4, f_{\rm sub})\cong L^{\rm Vir}_{c=-2} \otimes \mathcal H$ (this can be seen either directly from the OPEs by renormalizing the field \smash{$\widetilde W:= (k+2)W$}, or from the result of \cite[Section 5.4]{CFLN}), hence $\mathcal C_{-2} \cong L^{\rm Vir}_{c=-2} $. But $\mathcal N_{-1/2}(\mathfrak{sl}_2) \cong \mathcal W_{-2}(\mathfrak{sl}_3, f_{\rm prin})$, hence the parafermionic subalgebras do not coincide.

 If $c=0$, then we have $(k,s) = \bigl(-\frac{5}{2}, 1\bigr)$ or $(k,s) = \bigl(-\frac{7}{3}, 1\bigr)$. If $(k,s) = \bigl(-\frac{5}{2}, 1\bigr)$, then $\mathcal W_{-5/2}(\mathfrak{sl}_4, f_{\rm sub})$ is isomorphic to the Heisenberg vertex algebra $M(1)$ (cf.\ \cite{APV}) and $\mathcal N_{1}(\mathfrak{sl}_2) \cong \mathbb{C} \mathbbm{1}$. If $(k,s) = \bigl(-\frac{7}{3}, 1\bigr)$, then $W_{-7/3} (\mathfrak{sl}_4, f_{\rm sub})$ is isomorphic to the rank one lattice vertex algebra~$V_{2 \Z}$ (cf.\ \cite[Theorem~5.5.]{CFLN}), which implies $\mathcal C_{-7/3} \cong \mathcal N_{1}(\mathfrak{sl}_2) \cong \mathbb{C} \mathbbm{1}$.
\end{proof}

\begin{Remark}
	Two of the cases from Proposition \ref{pf_coincidence} are in fact collapsing levels: for $(k,s) = \bigl(-\frac{8}{3},0\bigr)$ it was shown in \cite{AMP-IMRN} that $\mathcal W_{-8/3}(\mathfrak{sl}_4, f_{\rm sub}) \cong \mathbb{C}\mathbbm{1}$, therefore the parafermionic subalgebras are 1-dimensional.
 If $(k,s) = \bigl(-\frac{5}{2}, 1\bigr)$, then $\mathcal W_{-5/2}(\mathfrak{sl}_4, f_{\rm sub})$ is isomorphic to the Heisenberg vertex algebra $M(1)$ (cf.\ \cite{APV}), implying that the parafermionic subalgebras are 1-dimensional.
\end{Remark}

\begin{Remark}
	The coincidences from Proposition \ref{pf_coincidence} can also be checked directly from the explicit expressions for OPEs of $\mathcal N_s(\mathfrak{sl}_2)$ from \cite{DLY}.
\end{Remark}

\subsection{Classification of KS dualities}

Using the Kazama--Suzuki duality of $L_s(\mathfrak{sl}_2)$ and $L^{N=2}_c$ (cf.\ \cite{A-IMRN}), where $c =\frac{3s}{s+2}$, we have the following.

\begin{Theorem}
	$L^{N=2}_c$ and $\mathcal W_k(\mathfrak{sl}_4, f_{\rm sub})$ are in Kazama--Suzuki duality if and only if $k=-1$ and $c=-15$ or $k=-\frac{7}{3}$ and $c=1$.
\end{Theorem}
\begin{proof}
Necessary conditions for $L^{N=2}_c$ and $\mathcal W_k(\mathfrak{sl}_4, f_{\rm sub})$ to be in Kazama--Suzuki duality are~(i) their corresponding coset subalgebras need to coincide and (ii) \smash{$(G^+)^2$} needs to be a singular vector in $\mathcal W^k(\mathfrak{sl}_4, f_{\rm sub})$ \big(in order for generators to satisfy the relations for the $L^{N=2}_c$ algebra\big),
 hence the only possibilities are those listed in Proposition \ref{pf_coincidence}. Among those, it follows easily from Proposition \ref{kriterij} that $(G^+)^2$ is a singular vector in $\mathcal W^k(\mathfrak{sl}_4, f_{\rm sub})$ only in the cases $(k,s) = \bigl(-1,-\frac{5}{3}\bigr)$ \big(which corresponds to $\mathcal W_{-1}(\mathfrak{sl}_4, f_{\rm sub})$ and $L^{{N=2}}_{c=-15}$\big), and $(k,s) = \bigl(-\frac{7}{3}, 1\bigr)$ \big(which corresponds to $\mathcal W_{-7/3}(\mathfrak{sl}_4, f_{\rm sub})$ and $L^{N=2}_{c=1}$\big).

Kazama--Suzuki duality between $L^{N=2}_{c=-15}$ and $\mathcal W_{-1}(\mathfrak{sl}_4, f_{\rm sub})$ will be studied in Section \ref{duality}. Let us now discuss the case of $\mathcal W_{-7/3}(\mathfrak{sl}_4, f_{\rm sub})$ and $L^{N=2}_{c=1}$.

We claim that there exist embeddings
\[
\Phi\colon\ L^{N=2}_{c=1} \rightarrow \mathcal{W}_{-7/3}(\mathfrak{sl}_4, f_{\rm sub}) \otimes \F_{-1}, \qquad \Phi^{\rm inv} \colon\ \mathcal{W}_{-7/3}(\mathfrak{sl}_4, f_{\rm sub}) \rightarrow L^{N=2}_{c=1} \otimes \F_{1},
\]
such that
 \begin{align*}
 	L^{N=2}_{c=1} \cong \operatorname{Com} (\mathcal{H}_1, \mathcal{W}_{-7/3}(\mathfrak{sl}_4, f_{\rm sub}) \otimes \F_{-1}), \qquad
 	\mathcal{W}_{-7/3}(\mathfrak{sl}_4, f_{\rm sub}) \cong \operatorname{Com} \bigl(\mathcal{H}_2, L^{N=2}_{c=1} \otimes \F_{1}\bigr),
 \end{align*}
 where $\mathcal{H}_1$ and $\mathcal{H}_2$ are Heisenberg vertex algebras.

But this follows easily from the fact that $\mathcal{W}_{-7/3}(\mathfrak{sl}_4, f_{\rm sub}) \cong \F_4$ (cf.\ \cite{CFLN}) and $ L^{N=2}_{c=1} \cong \F_3$ (cf.~\cite{K1}), where $\F_n = V_{\mathbb{Z}\alpha}$ is the lattice vertex algebra associated to the lattice $L= {\Z}\alpha $, $\langle \alpha , \alpha \rangle = n$.
\end{proof}

\section[The duality of W\_-1(sl\_4, f\_sub) and L\^{}N=2\_c=-15]{The duality of $\boldsymbol{\mathcal{W}_{-1}(\mathfrak{sl}_4, f_{\rm sub})}$ and $\boldsymbol{L^{N=2}_{c=-15}}$} \label{duality}

As we have seen in Section \ref{KS_gen}, Kazama--Suzuki duality between $L^{N=2}_c$ and $\mathcal W_k(\mathfrak{sl}_4, f_{\rm sub})$ can only occur if $k=-1$ and $c=-15$ or $k=-\frac{7}{3}$ and $c=1$.
In this section, we consider the case~${k=-1}$. Then the central charge of $\mathcal{W}^{k}(\mathfrak{sl}_4, f_{\rm sub})$ is $c = -15$. According to Lemma \ref{sing_G^+}, \smash{$(G^+)^2$} is singular vector in $\mathcal{W}^{k}(\mathfrak{sl}_4, f_{\rm sub})$, hence \smash{$(G^+)^2 = 0$} in $\mathcal{W}_{k}(\mathfrak{sl}_4, f_{\rm sub})$.

 We will show that $\mathcal{W}_{-1}(\mathfrak{sl}_4, f_{\rm sub})$ is the Kazama--Suzuki dual of the $N=2$ superconformal vertex algebra of central charge $-15$.

\subsection[Embedding of L\^{}N=2\_c=-15 into W\_-1(sl\_4, f\_ ub)]{Embedding of $\boldsymbol{ L^{N=2}_{c=-15}}$ into $\boldsymbol{\mathcal{W}_{-1}(\mathfrak{sl}_4, f_{\rm sub}) \otimes \F_{-1}}$}

First, we will show that the $N=2$ superconformal algebra $L^{N=2}_{c}$ for $c=-15$ can be realized as a subalgebra of $\mathcal{W}_{-1}(\mathfrak{sl}_4, f_{\rm sub}) \otimes \F_{-1}$, where $\F_{\pm 1}$ is the vertex algebra associated to the rank one lattice $L= {\mathbb{Z}}\varphi^{\pm}$, with $\big\langle \varphi^{\pm}, \varphi^{\pm} \big\rangle = \pm 1$ (cf.\ Section \ref{F_{-1}}).

 In this section, we shall use that fact that the universal superconformal $N=2$ vertex algebra~$V^{N=2}_{c}$ is simple for $c=-15$, since it is the Kazama--Suzuki dual of the universal affine vertex algebra $V^{k=-5/3}(\mathfrak{sl}_2)$ (which is also simple by \cite{GK}).

 Note that the maximal ideal in $\mathcal W^k(\mathfrak{sl}_4, f_{\rm sub})$ is invariant under the automorphism which maps $G^+$ to $G^-$. Since \smash{$(G^+)^2$} is a singular vector in $\mathcal{W}^{-1}(\mathfrak{sl}_4, f_{\rm sub}) $, we conclude that $(G^-)^2$ also belongs to the maximal ideal of $\mathcal{W}^{-1}(\mathfrak{sl}_4, f_{\rm sub}) $. Now we want to show that the maximal ideal is generated by these two vectors.
Let \smash{$I = \bigl\langle (G^+)^2, (G^-)^2 \bigr\rangle$} and define
\[
 \widetilde W = \mathcal{W}^{-1}(\mathfrak{sl}_4, f_{\rm sub})/I.
 \]

Let $H^{\perp} = J + \varphi^-$. Then for $n \ge 0$, we have
\[
H^{\perp}(n) H^{\perp} = \frac{1}{4} \delta_{n,1}\mathbbm{1}.
\]
Let $M_{H^{\perp}} (1)$ be the Heisenberg vertex algebra generated by $H^{\perp}$, and $M_{H^{\perp}} (1, s)$ the irreducible highest weight $M_{H^{\perp}} (1)$-module on which
$H^{\perp}(0) \equiv s \operatorname{Id}$.

\begin{Proposition} \label{ulaganje-1}\qquad
\begin{itemize}\itemsep=0pt
\item[$(1)$]There is a vertex algebra homomorphism $\Phi\colon L^{N=2}_{c=-15} \rightarrow \widetilde W \otimes \F_{-1}$ given by
	\begin{gather*}
		E = \frac{2}{3} G^+ \otimes e^{\varphi^-}, \qquad
		F = - G^- \otimes e^{-\varphi^-},\qquad
		H = - 4 J - 5 \varphi^-, \\
		T = L- 2{:}J J {:}-4{:}J\varphi^-{:} -\frac{5}{2}{:}(\varphi^-)^2{:}.
	\end{gather*}
	\item[$(2)$]
	Let $H^{\perp} = J + \varphi^-$. Then
\[
 \widetilde W \otimes \F_{-1} \cong \bigoplus_{n \in {\mathbb{Z}}} \sigma^n \bigl(L^{N=2}_{c=-15}\bigr) \otimes M_{H^\perp} (1, n).
 \]
	\item[$(3)$] We have
\[
\operatorname{Im} (\Phi) \cong \operatorname{Com} (M_{H^{\perp}} (1), \mathcal{W}_{-1}(\mathfrak{sl}_4, f_{\rm sub}) \otimes \F_{-1}) \cong L^{N=2}_{c=-15} .
\]
\item[$(4)$] $\widetilde W = \W_{-1} (\mathfrak{sl}_4, f_{\rm sub})$, i.e., $I$ is the maximal ideal in $\W^{-1}(\mathfrak{sl}_4, f_{\rm sub})$.
\end{itemize}
	
\end{Proposition}

\begin{proof}

It is easy to check that
\begin{gather*}
 H(0) E = E, \qquad H(0) F =-F, \qquad H(n) E = H(n) F =0, \qquad n >0, \\
 H(1) H = - 5 {\mathbbm 1} = \frac{c}{3} {\mathbbm 1},\qquad H(n) H =0, \qquad n >1.
\end{gather*}
 Let $L^{\perp} = L- \frac{2}{5} {:}J J{:}$ (cf.\ Appendix \ref{appendix-b}). We get that
\[
 T = L^{\perp} -\frac{1}{10} {:}H H{:}.
 \]
This implies that $T$ is a Virasoro vector of central charge $c=-15$. Clearly, we have that $H$ is a primary vector for $T$ of conformal weight $1$.

 Next we notice that $\bigl(G^\pm \bigr)^2 =0$ in $\widetilde W$. Direct calculation shows that
 \begin{gather*}
 E_{(0)} E = \frac{4}{9} {\bigl( G^+ \bigr)}^2 \otimes e^{2 \varphi^-} =0, \qquad F_{(0)}F = {( G^-)}^2 \otimes e^{-2 \varphi^-} =0, \\
 E_{(n)} E = F_{(n)} F =0, \qquad n\in {\mathbb Z}_{\ge 0},
\end{gather*}
and
\begin{gather*}
	E_{(2)} F = -\frac{2 (2 + k)(5 + 2k)(8 + 3k)}{3} {\mathbbm 1} =(-10) {\mathbbm 1} = \frac{2 c}{3} {\mathbbm 1}, \\
E_{(1)} F = -\frac{ 8(2 + k)(5 + 2k)}{3} J - \frac{2 (2 + k)(5 + 2k)(8 + 3k)}{3} \varphi^- = -8 J - 10 \varphi^- = 2 H, \\
 E_{(0)} F = - \frac{2 (2 + k)(5 + 2k)(8 + 3k)}{3}\cdot\frac{1}{2}\bigl({:}( \varphi^-)^2{:}+\partial\varphi^-\bigr)-\frac{ 8(2 + k)(5 + 2k)}{3} {:}J \varphi^-{:} \\
 \phantom{E_{(0)} F =}{}+\frac{2(k+2) \bigl((k+4)L-6{:}J^2{:}-2(2k+5) \partial J\bigr)}{3} = 2T + \partial H.
\end{gather*}

The above relations show that the even fields $H$, $ T$ and the odd fields $E$, $F$ satisfy the $\lambda$-bracket for the $N=2$ superconformal vertex algebra of central charge $c=-15$. Since \smash{$V^{N=2} _{c=-15}$} is simple, we conclude that $H$, $T$, $E$, $F$ generate a vertex subalgebra of $\widetilde W \otimes \F_{-1}$ isomorphic to~$L^{N=2}_{c=-15}$. This proves the assertion~(1).

 Let \smash{$ \overline W = \operatorname{Ker} _{ \widetilde W\otimes F_{-1}} H^{\perp}(0)$}. Then $\overline{W}$ is a vertex algebra which contains $ \operatorname{Im} (\Phi)\otimes M_{H^{\perp}} (1)$.

Let us prove the following claim.
\begin{Claim}\label{Claim1}
 $ \overline W $ is a simple vertex algebra generated by $\big\{ E,F,T,H, H^{\perp} \big\}$.
\end{Claim}

Let $U$ be the vertex subalgebra of $\overline W $ generated by $ \big\{E,F,T,H, H^{\perp} \big\}$. Then $U \cong \operatorname{Im} (\Phi) \otimes M_{H^{\perp}} (1)$. Since \smash{$\operatorname{Im} (\Phi) \cong L^{N=2}_{c=-15}$}, we conclude that
\smash{$U\cong L^{N=2}_{c=-15} \otimes M_{H^{\perp}} (1)$}, and it is therefore simple.

 For each $n \in {\Z}$, we consider the $U$-module $U^{(n)} = U. e^{n \varphi^-}$. One can show that
 $ U^{(n)} $ is isomorphic to the simple $U$-module obtained by the simple current construction
\[
 \bigl(U^{(n)}, Y^{(n)} (\cdot, z)\bigr): = (U , Y(\Delta(n \varphi^-,) \cdot, z)).
 \]

 Note that $ \varphi^- = -H - 4 H^{\perp}$, which implies that
 \[
 \Delta(\varphi^-, z) = \Delta( -H, z) \Delta\bigl(- 4 H^{\perp}, z\bigr).
 \]
 \begin{itemize}\itemsep=0pt
\item For a $ \operatorname{Im} (\Phi)$-module $M$, by applying the operator $\Delta(-nH, z)$, we get the module $\sigma^{n}( M)$.
\item Applying the operator $ \Delta\bigl( -4n H^{\perp}, z\bigr)$ on
 $ M_{H^{\perp}}(1)$, we get the $ M_{H^{\perp}}(1)$-module
$ M_{H^{\perp}} (1, n )$.
\end{itemize}

We get
$
U^{(n)}= \sigma^{n}( \operatorname{Im} (\Phi)) \otimes M_{H^{\perp}} (1,n )$.
Note that $ H^{\perp} (0) \equiv n \operatorname{Id} $ on $U^{(n)}$.
 Using the construction of H. Li from \cite{Li-ext}, we get that
\[
 \mathcal U = \bigoplus_{n \in \mathbb{Z}} U^{(n)}
 \]
is a vertex algebra and hence a vertex subalgebra of $ \widetilde W \otimes \F_{-1}$. But it is not hard to see that $ \mathcal U$ contains all generators of $ \widetilde W \otimes \F_{-1}$.

 Indeed, $e^{\varphi^-}\in U^{(1)}$, $e^{-\varphi^-}\in U^{(-1)}$ and
 \[
 G^+ = \frac{3}{2} E_{(-2)}e^{-\varphi^-}, \qquad G^- =- F_{(-2)}e^{\varphi^-}.
 \]

Since $G^+$ and $G^-$ (weakly) generate $ \widetilde W$, it follows that
$
 \mathcal U = \widetilde W \otimes \F_{-1}$.

This proves that $ U= \overline W \cong \operatorname{Im} (\Phi) \otimes M_{H^{\perp}} (1) $. Therefore, $\overline W$ is a simple vertex algebra. This proves Claim~\ref{Claim1}.

Now we get
\[
 \widetilde W \otimes \F_{-1} = \bigoplus_{n \in \mathbb{Z}} U^{(n)} = \bigoplus_{n \in \mathbb{Z}} \sigma^n \bigl(L^{N=2}_{c=-15}\bigr) \otimes M_{H^\perp} (1, n).
 \]
This implies the assertions (2) and (3).

As the right-hand side of decomposition in (2) is simple, it follows that $\widetilde W $ must be simple and therefore isomorphic to the unique simple quotient $\mathcal{W}_{-1}(\mathfrak{sl}_4, f_{\rm sub})$. Hence $I$ is the maximal ideal in $\mathcal{W}^{-1}(\mathfrak{sl}_4, f_{\rm sub})$.
This proves assertion (4).
\end{proof}

 Let $M_H(1)$ be the Heisenberg subalgebra of $L^{N=2}_{c=-15}$ generated by the Heisenberg field $H(z)$ and $M_J(1)$ the Heisenberg subalgebra of $\mathcal{W}_{k}(\mathfrak{sl}_4, f_{\rm sub}) $ generated by the Heisenberg field $J(z)$.
From the proof of Proposition~\ref{ulaganje-1}, we have the following corollary.

	\begin{Corollary} \label{pf} $\operatorname{Com} \bigl(M_H(1) , L^{N=2}_{c=-15} \bigr) \cong \operatorname{Com} (M_J(1), \mathcal{W}_{-1}(\mathfrak{sl}_4, f_{\rm sub}) )$. \end{Corollary}
	
\subsection[Embedding of W\_k(sl\_4, f\_sub) into L\^{}N=2\_c=-15]{Embedding of $\boldsymbol{\mathcal{W}_{k}(\mathfrak{sl}_4, f_{\rm sub})}$ into $\boldsymbol{L^{N=2}_{c=-15} \otimes \F_{1}}$ }

It remains to show that $\mathcal{W}_{-1}(\mathfrak{sl}_4, f_{\rm sub}) $ can be realized as a subalgebra of the tensor product of the $N=2$ superconformal algebra $L^{N=2}_{c}$ for $c=-15$ and the lattice vertex algebra $\F_1$.

Let $\overline{H} = H - \varphi^+$. Then for $n \ge 0$ we have $\overline H(n) \overline H = -4 \delta_{n,1}\mathbbm{1}$.
Let $M_{\overline{H}} (1)$ be the Heisenberg vertex algebra generated by $\overline H$, and $M_{\overline{H}} (1, s)$ the irreducible highest weight $M_{\overline{H}} (1)$-module on which
$\overline H(0) \equiv s \operatorname{Id}$.

\begin{Proposition} \label{ulaganje-2}
\quad
\begin{itemize}\itemsep=0pt
	\item[$(1)$] There is a vertex algebra homomorphism $\Phi^{\rm inv} \colon \mathcal{W}^{-1}(\mathfrak{sl}_4, f_{\rm sub}) \rightarrow L^{N=2}_{c=-15} \otimes \F_{1}$ given by
	\begin{gather*}
		G^+ = E\otimes e^{\varphi^+}, \qquad
		G^- = \frac{3}{2} F \otimes e^{-\varphi^+}, \qquad
		J = - \frac{1}{4} H + \frac{5}{4} \varphi^+, \\
		L = T + \frac{1} {10} {:}H H{:}+ \frac{2}{5} {:}\left( - \frac{1}{4} H + \frac{5}{4} \varphi^+\right)^2{:}, \qquad
		W = -\frac{3}{2} W^{N=2}_{c=-15}.
	\end{gather*}
	
\item[$(2)$] $\operatorname{Im} \bigl(\Phi^{\rm inv}\bigr) $ is isomorphic to the simple vertex algebra $ \mathcal{W}_{-1}(\mathfrak{sl}_4, f_{\rm sub})$.

\item[$(3)$]
 As a $ \mathcal{W}_{-1}(\mathfrak{sl}_4, f_{\rm sub})\otimes M_{\overline{H}} (1)$-module
\[
L^{N=2}_{c=-15} \otimes \F_{1} \cong \bigoplus_{n\in {\Z} }
 \psi^{-n} ( \mathcal{W}_{-1}(\mathfrak{sl}_4, f_{\rm sub})) \otimes M_{\overline{H}} (1, n).
 \]
 \item[$(4)$] We have
\[
 \operatorname{Com} \bigl(M_{\overline{H}}(1), L^{N=2}_{c=-15} \otimes \F_{1}\bigr) \cong \mathcal{W}_{-1}(\mathfrak{sl}_4, f_{\rm sub}).
 \]
 \end{itemize}
	
\end{Proposition}
\begin{proof}
(1)
 First we notice that
$L = T^{\perp} + \frac{2}{5}{:}JJ{:}$,
which implies that $L$ is a Virasoro vector of central charge $c=-15$. Note that
$
 L^{\perp} = L- \frac{2}{5}{:}JJ{:} = T^{\perp}$.
Clearly, $J$ is a primary vector of conformal weight $1$ with respect to $L$.
Moreover, since
$L = T + \frac{2}{5}{:}JJ{:} + \frac{1}{10} {:}H H{:} $, we conclude that $G^{\pm}$ are primary vectors of conformal weight $2$.
Direct calculation shows that
\begin{gather*}
 G^+(3) G^- = \frac{3}{2}(-E_{(2)}F) = \frac{3}{2}\left(- \frac{2 c}{3}\right) \mathbbm{1} =15\mathbbm{1} =(2 + k)(5 + 2k)(8 + 3k)\mathbbm{1}, \\
 G^+(2) G^- =\frac{3}{2}\left(-E_{(1)}F- \varphi^+(-1) E_{(2)}F\right) \\
 \phantom{G^+(2) G^- }{} = \frac{3}{2}\left(-2H -\frac{2c}{3} \varphi^+\right) = 12\left( - \frac{1}{4} H + \frac{5}{4} \varphi^+ \right) =12J = 4(2 + k)(5 + 2k)J.
 \end{gather*}

 We should show that
\[ G^+(1) G^-
 = - 3 L + 6 {:} JJ {:} + 6 \partial J = -3 L^{\perp} + \frac{24}{5} J^2 + 6 \partial J.
 \]
 Indeed
\begin{align*}
G^+(1) G^- &= \frac{3}{2}\left(-E_{(0)}F-\varphi^+(-1) E_{(1)}F-\frac{1}{2}\bigl({:}\bigl(\varphi^+\bigr)^2{:}+\partial\varphi^+\bigr)E_{(2)}F \right) \\
&= \frac{3}{2}\left( -2 T -\partial H -2{:}H\varphi^+{:} - \frac{1}{2}\bigl({:}\bigl(\varphi^+\bigr)^2{:}+\partial\varphi^+ \bigr)\frac{2c}{3}\right)\mathbbm{1} \\
& = -3 L^{\perp}+ \frac{3}{10} {:}H^2{:}- \frac{3}{2} \partial H -3 {:}H \varphi^+ {:} + \frac{15}{2} \bigl({:}\bigl(\varphi^+\bigr)^2{:} +\partial \varphi^+\bigr) \\
& = -3 L^{\perp} + \frac{24}{5} {:}J ^2{:} + 6\partial J.
 \end{align*}

 Next, we need to show that
 \begin{align}
G^{+}(0) G^{-} ={}&(k+2)\left(W + \frac{8(11k+32)}{3(3k+8)^2}{:}J^3{:}-\frac{4(k+4)}{3k+8}{:}L J{:} + 6{:}\partial JJ{:}\right. - \frac{k+4}{2}\partial L \nonumber\\
 &+ \left. \frac{4(3k^2+17k+26)}{3(3k+8)}\partial^2J\right) \nonumber \\
={}& W+ \frac{32}{25}{:}J^3{:}-\frac{12}{5}{:}JL^{\perp}{:}+\frac{24}{5} {:}\partial JJ{:}-\frac{3}{2}\partial L^{\perp}+2\partial^2J.\label{GpGm-OPE}
\end{align}

We have that
 \begin{align}
G^+(0) G^-={}& \frac{3}{2}\left(-E_{(-1)}F -\varphi^+(-1) E_{(0)}F - E_{(1)}F\cdot \frac{1}{2}\bigl({:}\bigl(\varphi^+\bigr)^2{:}+\partial \varphi^+\bigr) \right. \nonumber\\
&\left. - E_{(2)}F\cdot \frac{1}{6}\bigl({:} \bigl(\varphi^+ \bigr)^3 {:}+3{:}\varphi ^+ \partial \varphi^+{:} + \partial ^2 \varphi^+\bigr) \right) \nonumber\\
={}& \frac{3}{2}\left(-E_{(-1)}F -\varphi^+(-1) \left(2L^{\perp} - \frac{1}{5}{:}H^2{:} +\partial H\right) - 2H\cdot\frac{1}{2}\bigl({:}\bigl(\varphi^+\bigr)^2{:}+\partial \varphi^+\bigr) \right. \nonumber\\
&\left.- \frac{2c}{3}\cdot \frac{1}{6}\bigl({:}\bigl(\varphi^+\bigr)^3{:} + 3{:}\varphi^+ \partial \varphi^+{:} + \partial ^2 \varphi^+\bigr) \right) .  \label{GpGm-rel}
\end{align}

As the parafermionic subalgebras of $\mathcal W_{-1} (\mathfrak{sl}_4, f_{\rm sub})$ and $L^{N=2}_{c=-15} $ coincide (cf.\ Corollary \ref{pf}), it follows that the field $W$ coincides (up to normalization) with the parafermionic generator~$W^{N=2}_{c=-15}$, that is,
\begin{align}
	W	&= \nu \left({:}E F{:} - \partial L^{\perp} - \frac{3}{2c}{:}H\partial H{:} -\frac{6}{c}{:}HL^{\perp}{:}- \frac{1}{3}\partial^2H + \frac{3}{c^2}{:}H^3{:} \right) \nonumber \\
	&= \nu \left({:}E F{:} - \partial L^{\perp}+ \frac{1}{5}{:}\partial H H {:} + \frac{2}{5}{:}HL^{\perp}{:} -\frac{1}{3}\partial^2H - \frac{1}{75}{:}H^3{:} \right). \label{W}
\end{align}
Setting $\nu = -\frac{3}{2}$ and substituting \eqref{W} into \eqref{GpGm-OPE}, we obtain that \eqref{GpGm-rel} = \eqref{GpGm-OPE}.
This proves the assertion (1).

(2) Let us prove that $\operatorname{Im} \bigl(\Phi^{\rm inv}\bigr) $ is simple. Let \smash{$\overline{W} = \operatorname{Ker}_{ L^{N=2}_{c=-15} \otimes \F_{1}} \overline{H} (0)$}. It is clear that $\overline{W}$ is a simple vertex algebra which contains $ \operatorname{Im} \bigl(\Phi^{\rm inv}\bigr) \otimes M_{\overline{H}} (1)$. The simplicity of $ \operatorname{Im} \bigl(\Phi^{\rm inv}\bigr) $ follows from the following claim.

\begin{Claim}\label{Claim2}
$\overline W$ is generated by $\big\{ G^+, G^-, J, L, W, \overline{H} \big\}$.
\end{Claim}

Proof of Claim~\ref{Claim2} is analogous to the proof of Claim~\ref{Claim1} in Proposition \ref{ulaganje-1}. Let $U$ be the vertex subalgebra of $\overline W$ generated by $\big\{ G^+, G^-, J, L^U, W, \overline{H} \big\}$. Then clearly $U \cong \operatorname{Im} \bigl(\Phi^{\rm inv}\bigr) \otimes M_{\overline{H}} (1)$.

Let $ U^{(n)} $ be the $U$-module obtained by the simple current construction
\[ \bigl(U^{(n)}, Y^{(n)} (\cdot, z)\bigr): = \bigl(U , Y\bigl(\Delta\bigl(n \varphi^+\bigr) \cdot, z\bigr)\bigr).
\]

As in Proposition \ref{ulaganje-1}, from the formula
\[
\Delta\bigl(n\varphi^+, z\bigr) = \Delta( nJ, z) \Delta\left(\frac{n}{4} \overline{H}, z\right)
\]
we get
$U^{(n)}= \psi^{-n}\bigl( \operatorname{Im} \bigl(\Phi^{\rm inv}\bigr)\bigr) \otimes M_{\overline{H}} (1,n )$.

The rest of the proof follows analogously.
\end{proof}

We have proved the following.

\begin{Theorem} The vertex algebra $\mathcal{W}_{-1}(\mathfrak{sl}(4), f_{\rm sub})$ is the Kazama--Suzuki dual of the $N=2$ superconformal vertex algebra $L^{N=2}_{c=-15}$.
\end{Theorem}

\section[Classification of irreducible modules for W\_-1(sl\_4, f\_sub]{Classification of irreducible modules for $\boldsymbol{\mathcal{W}_{-1}(\mathfrak{sl}_4, f_{\rm sub})}$}

 In this section, we will classify irreducible highest weight modules for $\mathcal{W}_{-1}(\mathfrak{sl}_4, f_{\rm sub})$, using the realization from Section \ref{duality}.
 We shall first show the existence of irreducible highest weight modules using the realization of $\mathcal{W}_{-1}(\mathfrak{sl}_4, f_{\rm sub})$ as a subalgebra of $L_c ^{N=2} \otimes \F_1$. Next we classify the irreducible modules using the inverse Kazama--Suzuki construction.

\subsection[Irreducible modules for W\_-1(sl(4), f\_sub]{Irreducible modules for $\boldsymbol{\mathcal{W}_{-1}(\mathfrak{sl}(4), f_{\rm sub})}$}
Recall that we consider $\mathcal{W}_{-1}(\mathfrak{sl}_4, f_{\rm sub})$ as a graded vertex algebra where the gradation is defined by the operator $\widetilde L(0)$.

 Furthermore $(G^+)^2$ is singular vector in $\mathcal{W}^{-1}(\mathfrak{sl}_4, f_{\rm sub})$ (cf.\ Lemma \ref{sing_G^+}). This implies that $ L(x,y,z)_{\text{\rm top}}$ can be either $1$-dimensional or $2$-dimensional.

 \subsection[Modules for W\_-1(sl\_4, f\_sub) with 1-dimensional top spaces]{Modules for $\boldsymbol{\mathcal{W}_{-1}(\mathfrak{sl}_4, f_{\rm sub})}$ with $\boldsymbol{1}$-dimensional top spaces}

 Let $L_c[h,q]$ be the irreducible highest weight $L_c ^{N=2}$-module, generated by the highest weight vector $v_{h,q}$ such $c=-15$ and
 $H(0) v_{h,q} = h v_{h,q}$, $ T(0) v_{h,q} = q v_{h,q}$.

 Let $w_{h,q} = v_{h,q} \otimes \mathbbm{1} \in L_c[h,q] \otimes \F_1$. We need the following result for $c=-15$.

 \begin{Lemma} \label{W(n)w}
 Let $W= W^{N=2}_{c}$. Let $W(z) = \sum_{n\in {\Z}} W(n) z^{-n-3}$.
 Then for $n \in {\Z}_{\ge 0}$
 \[
 W(0) w_{h,q} = \nu \left( 2 q -\frac{6}{c} ( q h + h) - \frac{2 (c-9)}{3c} h + \frac{6}{c^2}h^3 \right)w_{h, q}.
 \]
 In particular, when $c=-15$, we get for all $n \in {\Z}_{\ge 0}$
 \[
 W(n ) w_{h,q} = -\frac{1}{25} \delta_{n,0} ( h+5) \bigl(h^2 - 5h + 15 q\bigr) w_{h,q} .
 \]
 \end{Lemma}
 \begin{proof}
Using the formula for normal ordered fields, we see that $({:} E F{:}) _{(2)}w_{h,q} =0 $.
Then from the expression \eqref{W}, we easily get $W(n) w_{h,q} =0$ for $n >0$ and
\[
 W(0) w_{h,q} = \nu \left( 2 q -\frac{6}{c} ( q h + h) - \frac{2 (c-9)}{3c} h + \frac{6}{c^2}h^3 \right)w_{h, q}.
 \]
The proof follows.
 \end{proof}

\begin{Lemma} \label{hw_Wk}
The vector $w_{h,q}$ is a highest weight vector for the $\mathcal{W}_{-1}(\mathfrak{sl}_4, f_{\rm sub})$-module $L_c[h,q]\otimes \F_1$ such that
\begin{gather*}
 J(0) w_{h,q} = -\frac{h}{4} w_{h,q}, \qquad
 L(0) w_{h,q} = \left(q+\frac{1}{8}h^2\right) w_{h,q}, \\
 W(0) w_{h,q} = -\frac{1}{25} (h+5) \bigl(h^2-5 h+15 q\bigr) w_{h,q}.
 \end{gather*}
 \end{Lemma}
\begin{proof}
Follows from Lemma \ref{W(n)w} or Lemma~\ref{lemma5.3} (proof is given in the appendix).
\end{proof}
	
\begin{Lemma}\label{lemma5.3} The projection of $W$ in $A\bigl(L_c^{N=2}\bigr)$ is given by
\[
[W] = -\frac{1}{25} ([H]+5) \bigl([H]^2-5 [H]+15[T]\bigr).
\]
 \end{Lemma}

 \begin{Proposition}\label{dim1-mod}
	Let $L(x,y,z)$ be an irreducible $\mathcal{W}_{k=-1}(\mathfrak{sl}_4, f_{\rm sub})$-module such that the top level $L(x,y,z)_{\text{\rm top}}$ is $1$-dimensional. Then for $(x,y,z) \in \mathbb{C}^3$ it holds that
\[
 g_1(x,y,z): = -6 x^2 + \frac{56}{25}x^3 + 4x - \frac{12}{5}xy + 3y + z = 0.
 \]
\end{Proposition}

\begin{proof}
Assume that the top level $L(x,y,z)_{\text{\rm top}}$ is $1$-dimensional. Then $G^-(0)v_{x,y,z}$ is a singular vector, or equivalently $\bigl[\bigl[G^+\bigr], [G^-]\bigr]v_{x,y,z} = 0$.

The statement follows from the following relation in the Zhu algebra $A\bigl(\mathcal{W}^{k}(\mathfrak{sl}_4, f_{\rm sub})\bigr)$ (proof is given in Appendix~\ref{appA}).
	\end{proof}

\begin{Lemma} 
		In the Zhu algebra $A(\mathcal{W}^{k}(\mathfrak{sl}_4, f_{\rm sub}))$ it holds that
	\begin{align*}
		\bigl[\bigl[G^+\bigr], [G^-]\bigr] ={}&\bigl[G^+(0)G^-\bigr] \\
={}& (k+2) \left(\frac{4 [J] \bigl(6 k^2+k (31-3 [L])-12[L]+40\bigr)}{9 k+24}+\frac{8 (11 k+32) [J] ^3}{3 (3 k+8)^2}\right.\\
&\left.+(k+4)[L]-6 [J] ^2+[W]\right).
	\end{align*}
For $k=-1$, we have
\begin{equation} \label{G^+G-}
\bigl[\bigl[G^+\bigr], [G^-]\bigr]= -6 [J] ^2 + \frac{56}{25}[J] ^3 + 4[J] - \frac{12}{5}[J] [L] + 3[L] + [W].
\end{equation}
\end{Lemma}

\begin{Corollary} \label{crit-1_dim}
 Assume that $L(x,y,z)$ is an irreducible $\mathcal{W}_{k=-1}(\mathfrak{sl}_4, f_{\rm sub})$-module such that the top level $L(x,y,z)_{\text{\rm top}}$ is $1$-dimensional.
 Then $L(x,y,z)$ can be realized as a subquotient of an~${L^{N=2}_c \otimes \F_{1}}$-module, viewed as $\mathcal{W}_{k=-1}(\mathfrak{sl}_4, f_{\rm sub})$-module.
\end{Corollary}

\begin{proof}
	The statement follows from Lemma \ref{hw_Wk} and Proposition \ref{dim1-mod}. Let $x= -\frac{h}{4}$, $ y = q+\frac{1}{8}h^2$ and substitute into the relation \eqref{G^+G-}. Consider $L_c[h,q] \otimes \F_1$ as $\mathcal{W}_{k=-1}(\mathfrak{sl}_4, f_{\rm sub})$-module. Then~\smash{$\widetilde L(x,y,z) = \mathcal{W}_{k=-1}(\mathfrak{sl}_4, f_{\rm sub}). (w_{h,q} \otimes {\bf 1}) $} is a highest weight $\mathcal{W}_{k=-1}(\mathfrak{sl}_4, f_{\rm sub})$-module whose top component is $1$-dimensional. In particular, its simple quotient $L(x,y,z)$ is irreducible~${\mathcal{W}_{k=-1}(\mathfrak{sl}_4, f_{\rm sub})}$-module with $1$-dimensional top space.
\end{proof}

\subsection[Modules for W\_-1(sl\_4, f\_sub) with 2-dimensional top spaces]{Modules for $\boldsymbol{\mathcal{W}_{-1}(\mathfrak{sl}_4, f_{\rm sub})}$ with $\boldsymbol{2}$-dimensional top spaces}

We consider again $L_{c} [h,q] \otimes \F_1$. Let $w_2 = w_{h,q} \otimes e^{\varphi^+}$. Then
\begin{gather*}
	G^+ (0) w_2 = E(-1/2) w_{h,q} \otimes e^{\varphi^+}_{(-1)} e^{\varphi^+} =0, \\
 G^- (0) w_2= F(1/2) w_{h,q} \otimes e^{-\varphi^+}_{(0)} e^{\varphi^+} = F(1/2) w_{h,q} \otimes {\mathbbm 1}, \\
J(0) w_{2} = \left(- \frac{1}{4} H(0) + \frac{5}{4} \varphi^+(0) \right) w_2 = - \frac{h - 5}{4} w_2 =: x w_2 , \\
 L(0) w_{2} = \left( T + \frac{1} {8} {:} H^2{:} - \frac{1}{4} {:} H \varphi^+{:} + \frac{5}{8} {:}\bigl(\varphi^+\bigr) ^2{:} \right)_{(1)} w_2
=\left( q + \frac{h^2 - 2h+ 5}{8}\right) w_2 =: y w_2, \\
 W(0) w_2 = - \frac{1}{25} (h+5) \bigl(h^2 - 5 h + 15 q\bigr) w_2 =: z w_2.
\end{gather*}

Clearly, $(x,y,z) \in {\C}^3$ defined above give all zeros of the polynomial
\[
g_2(x,y,z) = z+ \frac{1}{25} (-5 + 2 x) \bigl(75 - 80 x + 28 x^2 - 30 y\bigr).
\]
 We get the following.

\begin{Lemma} \label{konstr-2}
Assume that
$g_2(x,y,z) = 0$. Then $L(x,y,z)$ is an irreducible highest weight $\mathcal{W}_{-1}(\mathfrak{sl}_4, f_{\rm sub})$-module such that $\dim L(x,y,z)_{\rm top} \le 2$.
\end{Lemma}
\begin{proof}
 Assume that $(x,y, z)$ is a zero of the equation $g_2(x,y,z) = 0$. Then it can be written in a form
\[
 (x,y,z) = \left(-\frac{h-5}{4}, q+ \frac{h^2 - 2 h + 5}{8}, -\frac{1}{25} (h+5) \bigl(h^2 - 5h + 15 q\bigr)\right)
 \]
for suitable
$(h,q) \in {\C} ^2$. The above computation shows that
$
\widetilde L(x,y,z) = \mathcal W_{-1} (\mathfrak{sl}_4, f_{\rm sub}). w_2
$
is a~highest weight $\mathcal W_{-1} (\mathfrak{sl}_4, f_{\rm sub})$-module whose top space is at most $2$-dimensional. Let $L(x,y,z)$ be the irreducible quotient of $\widetilde L(x,y,z)$. Then $L(x,y,z)$ is an irreducible $\mathcal W_{-1} (\mathfrak{sl}_4, f_{\rm sub})$-module such that $\dim L(x,y,z)_{\rm top} { \le } 2$.
\end{proof}

\subsection[Classification of irreducible W\_-1(sl\_4, f\_sub)-modules]{Classification of irreducible $\boldsymbol{\mathcal{W}_{-1}(\mathfrak{sl}_4, f_{\rm sub})}$-modules}

 Let
\begin{gather*}
 S_{1} = \left\{ \left(-\frac{h}{4}, q+ \frac{h^2}{8}, -\frac{1}{25} (h+5) \bigl(h^2 - 5h + 15 q\bigr)\right),\, (h,q) \in {\C} ^2\right\}, \\
 S_{2} = \left\{ \left(-\frac{h{-}5}{4}, q+ \frac{h^2 - 2 h + 5}{8}, -\frac{1}{25} (h+5) \bigl(h^2 - 5h + 15 q\bigr)\right), \,(h,q) \in {\C} ^2\right\} .
\end{gather*}

We have seen earlier that $S_i = \{ (x,y,z) \mid g_i (x,y,z) =0\}$, $i=1,2$. Now we have our main result.

\begin{Theorem} \label{klas-1}
The set $\{L(x,y,z) \mid (x,y,z) \in S_1 \cup S_2 \}$ provides a complete list of irreducible, highest weight $\mathcal{W}_{k=-1}(\mathfrak{sl}_4, f_{\rm sub})$-modules. Moreover,
\begin{gather*}
 \dim L(x,y,z)_{\rm top} = 1 \iff (x,y,z) \in S_1, \\ \dim L(x,y,z)_{\rm top} = 2 \iff (x,y,z) \in S_2 \setminus S_1.
 \end{gather*}
\end{Theorem}

\begin{proof}
We already proved in Corollary \ref{crit-1_dim} and Lemma \ref{konstr-2} that $L(x,y,z)$, $(x,y,z) \in S_1 \cup S_2$, are irreducible $\mathcal{W}_{-1}(\mathfrak{sl}_4, f_{\rm sub})$-modules. Now we shall see that these modules give all irreducible, highest weight $\mathcal{W}_{k=-1}(\mathfrak{sl}_4, f_{\rm sub})$-modules.

Since ${:}\bigl(G^{\pm}\bigr)^2{:}$ belong to the maximal ideal in $\mathcal{W}^{k=-1}(\mathfrak{sl}_4, f_{\rm sub})$, it follows that ${:}\bigl(G^{\pm}\bigr) ^2{:} =0$ in $\mathcal{W}_{k=-1}(\mathfrak{sl}_4, f_{\rm sub})$ and hence $\dim L(x,y,z)_{\rm top} \le 2$.

Assume first that $\dim L(x,y,z)_{\rm top} = 1$. Then using the homomorphism
\[
\Phi\colon\ L^{N=2}_{c=-15} \rightarrow \mathcal{W}_{-1}(\mathfrak{sl}_4, f_{\rm sub}) \otimes \F_{-1}
\]
 from Proposition \ref{ulaganje-1}, we get that $L(x,y,z) \otimes \F_{-1}$ is a $L^{N=2}_{c=-15}$-module. Then $w_1 = v_{x,y,z} \otimes {\mathbbm 1}$ is a highest weight vector for the action of $ L^{N=2}_{c=-15}$ with highest weight $ (h,q) = \bigl(-4 x, y -2 x^2\bigr)$. Since $W = -\frac{3}{2} W^{N=2}_{c=-15} $, we get that
\[
W(0) w_1 = zw_1= -\frac{1}{25} (h+5) \bigl(h^2 - 5h + 15 q\bigr) w_1.
\]
 Therefore, $(x,y,z) \in S_1$.

Assume next that $\dim L(x,y,z)_{\rm top} = 2$ and consider again $L(x,y,z) \otimes \F_{-1}$ as a $L^{N=2}_{c=-15}$-module. Let \smash{$w_2 = v_{x,y,z} \otimes e^{\varphi^-}$}. It follows that $w_2$ is a highest weight vector for the action of~$ L^{N=2}_{c=-15}$ with highest weight $ (h,q) = \bigl(-4 x+5, y -2 x^2 + 4 x - \frac{5}{2}\bigr)$. Then
\[
 x= -\frac{h-5}{4}, \qquad y = q+ 2x^2 - 4 x + \frac{5}{2} = q + \frac{1}{8} h^2 - \frac{1}{4} h + \frac{5}{8}.
 \]

As above, the action of $W(0)$ is given by
\[
W(0) w_2 = -\frac{1}{25} (h+5) \bigl(h^2 - 5h + 15 q\bigr) w_2.
\]
 Hence, $(x,y,z) \in S_2$.
The proof follows.
\end{proof}

 \appendix

\section[Relations in the Zhu algebra A(W\^{}k(sl\_4, f\_sub))]{Relations in the Zhu algebra $\boldsymbol{ A(\mathcal{W}^{k}(\mathfrak{sl}_4, f_{\rm sub}))}$}\label{appA}

The vertex algebra $\mathcal{W}^{k}(\mathfrak{sl}_4, f_{\rm sub})$ is generated by fields $J$, $L$, $G^+$, $G^-$, $W$ of conformal weight 1, 2, 1, 3, 3 with respect to the grading operator $\widetilde L(0)$.
	Hence the Zhu algebra $A\bigl(\mathcal{W}^{k}(\mathfrak{sl}_4, f_{\rm sub})\bigr)$ is generated by $ [G^+]$, $ [G^- ]$, $ [J ]$, $ [L]$, $ [W]$.

\begin{Lemma} \label{zhu-1}
		In the Zhu algebra $A\bigl(\mathcal{W}^{k}(\mathfrak{sl}_4, f_{\rm sub})\bigr)$ it holds that
	\begin{align*}
		\bigl[\bigl[G^+\bigr], [G^-]\bigr] ={}&\bigl[G^+(0)G^-\bigr] \\
		 ={}& (k+2) \left(\frac{4 x \bigl(6 k^2+k (31-3 y)-12 y+40\bigr)}{9 k+24}+\frac{8 (11 k+32) x^3}{3 (3 k+8)^2}\right.\\
&\left.+(k+4) y-6 x^2+z\right).
	\end{align*}
For $k=-1$, we have
\[
 \bigl[\bigl[G^+\bigr], [G^-]\bigr] = -6 [J]^2 + \frac{56}{25}[J]^3 + 4[J] - \frac{12}{5}[J] [L] + 3[L] + [W].
 \]
\end{Lemma}
		
	\begin{proof}
	Using the commutator formula (cf.\ \cite{Z})
	\[
 [a]*[b] - [b]*[a] = \text{Res}_z (z+1)^{\deg a -1} Y(a,z) b,
 \]
 we have that
 \[
 \bigl[\bigl[G^+\bigr], [G^-]\bigr] = \text{Res}_z (z+1)^{\deg G^+ -1} G^+(z)G^- = \bigl[G^+(0)G^-\bigr].
 \]
	
From the OPEs for the subregular $\mathcal W$-algebra $\mathcal{W}^{k}(\mathfrak{sl}_4 f_{\rm sub})$, we have that
		\begin{align*}
		G^+(0)G^- ={}& (k+2)\left(W +\frac{8(k+11)}{3(3k+8)^2}{:}J^3{:}-\frac{4(k+4)}{3k+8}{:}LJ{:} + 6{:}\partial JJ{:}\right. \\
	 &+ \left.\frac{k+4}{2}\partial L + \frac{4\bigl(3k^2+17k+26\bigr)}{3(3k+8)}\partial ^2J\right) .
		\end{align*}	
Let us compute the projection of these elements in the Zhu algebra $A\bigl(\mathcal{W}^{k}(\mathfrak{sl}_4, f_{\rm sub})\bigr)$.
Since
\[
[J(-i_1 -1)\cdots J(-i_n -1)\mathbbm{1}] = (-1)^{i_1 + \dots + i_n}[(J(-1)^n\mathbbm{1}] = (-1)^{i_1 + \dots + i_n}x^n,
\]
 we have the following relations:
\begin{enumerate}\itemsep=0pt
	\item[(1)] $\bigl[{:}J^3{:}\bigr]=\bigl(J(-1)^3\mathbbm{1}\bigr] = [J]^3$,
	\item[(2)] $[{:}\partial JJ{:}] = [J(-2)J(-1)\mathbbm{1}] = -\bigl[J(-1)^2\mathbbm{1}\bigr] = -[J]^2$,
	\item[(3)] $\bigl[\partial ^2J\bigr] = 2[J(-3)\mathbbm{1}] = 2[J]$.
\end{enumerate}

Next, we claim that
\begin{enumerate}
 \item[(4)] $[\partial L] = [L(-3)\mathbbm{1}] = -2[L]$.
\end{enumerate}	
Since
$(L(-n-2) + 2L(-n-1)+ L(-n))v \in O(V)$
(cf.\ \cite{Z}),	it follows that $[L(-3)\mathbbm{1}] = -2[L(-2)\mathbbm{1}] = -2[L] $.
\begin{enumerate}
 \item[(5)] $[{:}LJ{:}] = [J][L]+[J]$.
\end{enumerate}		
	
We have
\begin{align*}
		[L]*[J] &= \text{Res}_z \frac{(1+z)^2}{z}L(z)J
		 = \left(\frac{1}{z} + 2 + z\right)\left(\sum L(n)z^{-n-2}\right)J \\
		&= [L(-2)J(-1)\mathbbm{1}] + 2[L(-1)J(-1)\mathbbm{1}] + [L(0)J(-1)\mathbbm{1}] \\
		&= [L(-2)J(-1)\mathbbm{1}] +[L(-1)J(-1)\mathbbm{1}] \\
		&= [L(-2)J(-1)\mathbbm{1}] +[J(-2)\mathbbm{1}] = [L(-2)J(-1)\mathbbm{1}] -[J].
\end{align*}

	Combining these relations and evaluating at $k=-1$ we obtain that
\[
 \bigl[G^+(0)G^-\bigr]= -6 [J]^2 + \frac{56}{25}[J]^3 + 4[J] - \frac{12}{5}[J] [L] + 3[L] + [W].\tag*{\qed}
 \] \renewcommand{\qed}{}
	\end{proof}

 \begin{Lemma} The projection of $W$ in $A\bigl(L_c^{N=2}\bigr)$ is given by
\[
[W] = -\frac{1}{25} ([H]+5) \left([H]^2-5 [H]+15[T]\right).
\]

 \end{Lemma}

 \begin{proof}
 	From the proof of Proposition \ref{ulaganje-2}, we have that the field $W \in \mathcal{W}^{k}(\mathfrak{sl}_4, f_{\rm sub})$ can be realized in $L^{N=2}_{c=-15} \otimes \F_{1}$ as
 	\begin{align*}
 W= -\frac{3}{2} W^{N=2}_{c=-15}
= -\frac{3}{2}\left({:}EF{:} - \partial L^{\perp}+ \frac{1}{5}{:}\partial H H {:} + \frac{2}{5}{:}HL^{\perp}{:} -\frac{1}{3}\partial^2H - \frac{1}{75}{:}H^3{:} \right).
 \end{align*}
 	
 	Analogously to the proof of Lemma \ref{zhu-1}, we compute the projection of the elements into the Zhu algebra $A\bigl(L_c^{N=2}\bigr)$:
 	\begin{enumerate}\itemsep=0pt
	\item[(1)] $\bigl[{:}H^3{:}\bigr]=\bigl[H(-1)^3\mathbbm{1}\bigr] = [H]^3$,
	\item[(2)] $[{:}\partial HH{:}] = [H(-2)H(-1)\mathbbm{1}] = -\bigl[H(-1)^2\mathbbm{1}\bigr] = -[H]^2$,
	\item[(3)] $\bigl[\partial ^2H\bigr] = 2[H(-3)\mathbbm{1}] = 2[H]$.
\end{enumerate}
As $ L^{\perp} = T + \frac{1}{10}{:}HH{:}$, we have
\begin{enumerate}\itemsep=0pt
 \item[(4)] $\bigl[\partial L^{\perp}\bigr] = \bigl[\partial T + \frac{1}{10}\partial ({:}HH{:})\bigr] = [T(-3)\mathbbm{1}] + \frac{1}{10}[2H(-2)H(-1)\mathbbm{1}] =-2[T] -\frac{1}{5}[H]^2 $.
\end{enumerate}	
We claim that
\begin{enumerate}\itemsep=0pt
 \item[(5)] $\bigl[{:}HL^{\perp}{:}\bigr] = [H][T]+\frac{1}{10}[H]^3$.
\end{enumerate}		
	
We have $\bigl[{:}HL^{\perp}{:}\bigr] = \bigl[{:}HT{:} + \frac{1}{10}{:}H^3{:}\bigr]$. Let us compute $[{:}HT{:}]$
\begin{align*}
		[H]*[T] &= \text{Res}_z \frac{(1+z)^1}{z}H(z)T = \left(\frac{1}{z} + 1\right)\left(\sum H(n)z^{-n-1}\right)T \\
		&= [H(-1)T(-2)\mathbbm{1}] + [H(0)T(-2)\mathbbm{1}] = [HT].
\end{align*}

Next, we have
\begin{enumerate}\itemsep=0pt
\item[(6)] $[E_{(-1)}F] = 0$.
\end{enumerate}
This follows from
	\begin{align*}
		[E_{(-1)} \mathbbm{1}] \circ [F_{(-1)} \mathbbm{1}] &= 0 = \text{Res}_z E (z)\frac{(1+z)^{{\frac{1}{2}} - \frac{1}{2}}}{z}F_{(-1)} \mathbbm{1} \\
		&= \text{Res}_z \frac{1}{z}\left(\sum E_{(n)}z^{-n-1}\right)F
		= [E_{(-1)}F].\tag*{\qed}
		\end{align*}\renewcommand{\qed}{}
\end{proof}

 \section[Operator product expansions for W\^{}k(sl\_4, f\_sub)]{Operator product expansions for $\boldsymbol{\mathcal{W}^{k}(\mathfrak{sl}_4, f_{\rm sub})}$}
 \label{appendix-b}

 The subregular $\mathcal{W}$-algebra $\mathcal{W}^{k}(\mathfrak{sl}_4, f_{\rm sub})$ is isomorphic to the Feigin--Semikhatov algebra $W^{(2)}_4$. It is generated by the even fields $J(z)$, $L(z)$, $G^{\pm}(z)$, $W(z)$ satisfying the following OPEs (cf.~\cite{FS}):\looseness=1
 \begin{gather*}
 J(x)J(y) \sim \frac{3k+8}{4(z-w)^{2}}, \qquad
 J(z)G^{\pm}(w) \sim \pm \frac{G^{\pm}(w)}{z-w} , \enskip \qquad G^{\pm}(z)G^{\pm}(w) \sim0, \\
	L(z)G^{\pm}(w) \sim \frac{2G^{\pm}(w)}{(z-w)^{2}} + \frac{\partial G^{\pm}(w)}{(z-w)}, \qquad
	L(z)J(w) \sim \frac{J(w)}{(z-w)^{2}} + \frac{\partial J(w)}{(z-w)}, \\
	L(z)L(w) \sim - \frac{c_k}{2(z-w)^{4}}+ \frac{2L(w)}{(z-w)^{2}}+ \frac{\partial L(w)}{z-w}, \\
	L(z) W(w) \sim \frac{3 W(w)}{(z-w)^{2}} + \frac{ \partial W(w)}{(z-w)}, \qquad J(z) W(w) \sim 0,
\\
	G^+(z)G^-(w) \sim \frac{(k+2)(2k+5)(3k+8)\mathbbm{1}}{(z-w)^{4}}+ \frac{4(k+2)(2k+5)J(w)}{ (z-w)^{3}} \\
	\phantom{	G^+(z)G^-(w) \sim }{} + (k+2) \frac{6 {:}JJ{:}(w) +2(2k+5)\partial J(w)-(k+4)L(w)}{(z-w)^{2}} \\
		\phantom{	G^+(z)G^-(w) \sim }{} + (k+2)\left( W(w) + \left(\frac{8(k+11)}{3(3k+8)^2}{:}J^3{:}(w)-\frac{4(k+4)}{3k+8}{:}L(w)J(w){:}\right.\right.\\
\left.\left.\phantom{	G^+(z)G^-(w) \sim }{}+ 6{:}\partial J(w)J(w){:}	 + \frac{k+4}{2}\partial L(w)\right.\right. \\
\left.\left.\phantom{	G^+(z)G^-(w) \sim }{}+ \frac{4(3k^2+17k+26)}{3(3k+8)}\partial ^2J(w)\right)\right)(z-w)^{-1},
\\
	 W(z)G^{\pm}(w) \sim \pm\frac{2(k+4)(3k+7)(5k+16)G^{\pm}(w)}{(3k+8)^2(z-w)^{3}}+ \left(\pm\frac{3(k+4)(5k+16)\partial G^{\pm}(w)}{2(3k+8)} \right. \\
\phantom{ W(z)G^{\pm}(w) \sim }{}	 \left. -\frac{6(k+4)(5k+16){:}J(w)G^{\pm}(w){:} }{(3k+8)^2}\right)(z-w)^{-2} \\
\phantom{ W(z)G^{\pm}(w) \sim }{}+\left(-\frac{8(k+4)(k+3){:}J(w)\partial G^{\pm}(w){:}}{(k+2)(3k+8)} \right. \\
\phantom{ W(z)G^{\pm}(w) \sim }{}	 \left. - \frac{4(k+4)(3k^2+15k+16){:}\partial J(w)G^{\pm}(w){:}}{(k+2)(3k+8)^2}\right.\\
 \phantom{ W(z)G^{\pm}(w) \sim }{}\left. \pm \frac{(k+4)(k+3)\partial ^2G^{\pm}(w)}{(k+2)}\mp \frac{2(k+4)^2{:}L(w)G^{\pm}(w){:}}{(k+2)(3k+8)}\right. \\
	 \phantom{ W(z)G^{\pm}(w) \sim }{} \left. \pm \frac{4(k+4)(5k+16){:}J(w)^2G^{\pm}(w){:}}{(k+2)(3k+8)^2} \right)(z-w)^{-1} , \\
	 W(z)W(w) \sim \frac{2(k+4)(2k+5)(3k+7)(5k+16)\mathbbm{1}}{(3k+8)(z-w)^{6}} -\frac{3(k + 4)^2( 5 k + 16)}{(3k+8)(z-w)^{4}}L^{\perp}(w) \\
\phantom{W(z)W(w) \sim}{}	 -\frac{3(k + 4)^2( 5 k + 16)}{2(3k+8)(z-w)^{3}}\partial L^{\perp}(w) \\
\phantom{W(z)W(w) \sim}{}+\left( -\frac{3(k + 4)^2( 5 k + 16)(12k^2 + 59k + 74)}{4(3k+8)(20k^2 + 93k + 102)}\partial^2L ^{\perp}(w)\right. \\
\phantom{W(z)W(w) \sim}{}	 + \left. \frac{8(k + 4)^3( 5 k + 16)}{(3k+8)(20k^2 + 93k + 102)}{:}L^{\perp}(w)L^{\perp}(w){:} + 4 (k+4) \Lambda (w)\right) (z-w)^{-2} \\
\phantom{W(z)W(w) \sim}{}	 + \biggl( -\frac{(k + 4)^2( 5 k + 16)(12k^2 + 59k + 74)}{6(3k+8)(20k^2 + 93k + 102)}\partial ^3L^{\perp}(w) \\
\phantom{W(z)W(w) \sim}{}	 +   \frac{8(k + 4)^3( 5 k + 16)}{(3k+8)(20k^2 + 93k + 102)}{:}\partial L^{\perp}(w)L^{\perp}(w){:} \\
\phantom{W(z)W(w) \sim}{}  + 2 (k+4) \partial \Lambda(w)\biggr)(z-w)^{-1},
\end{gather*}
 where
 \begin{align*}
 	(k+2)^2\Lambda = {}&{:} G^+ G^-{:} + (k+2) \left(-\frac{\partial W(z)}{2} -\frac{4 {:}W(z)J(z){:}}{3k+8}\right. \\
 &\left.+ \frac{(k+2)(k+4)\bigl(6k^2+33k+46\bigr)}{2(3k+8)\bigl(20k^2+93k+102\bigr)}\partial^2L^{\perp}(z)\right. \\
 	 &- \left. \frac{(k+4)^2(11k+26)}{2(3k+8)\bigl(20k^2+93k+102\bigr)}{:}L^{\perp}(z)L^{\perp}(z){:} + \frac{2(k+4)}{3k+8}\partial {:}L^{\perp}(z)J(z){:} \right.\\
 	 &+ \left. \frac{8(k+4)}{(3k+8)^2}{:}L^{\perp}(z)J(z)J(z){:} - \frac{2k+5}{3k+8}\left(\frac{8}{3}{:}\partial ^2J(z)J(z){:} + 2{:}\partial J(z)\partial J(z){:}\right. \right. \\
 	 &+ \left. \left. \frac{16}{3k+8}{:}\partial J(z)J(z)J(z){:} + \frac{32}{3(3k+8)^2}{:}J(z)^4{:} + \frac{3k+8}{6}\partial ^3J(z)\right) \right),
 \end{align*}
 $L^{\perp} = L-\frac{2}{3k+8}{:}JJ{:}$ and $c_k = -\frac{(3k+8)(8k+17)}{(k+4)}$.

 \subsection*{Acknowledgements}
We thank Shigenori Nakatsuka for interesting discussions on Kazama--Suzuki duality and Andrej Dujella for help in numerical calculations in Section~\ref{KS_gen}. We also would like to thank the referees for their valuable comments.
The authors are partially supported by the Croatian Science Foundation under the project IP-2022-10-9006
 and by the project ``Implementation of cutting edge research and its application as part of the Scientific Center of Excellence QuantiXLie'', PK.1.1.02, European Union, European Regional Development Fund. A part of this work was done while D.A. was staying at RIMS, Kyoto University in September 2024.

\pdfbookmark[1]{References}{ref}
\LastPageEnding

\end{document}